\documentclass[3p,10pt,a4paper,twoside,fleqn,sort&compress]{elsarticle}
\usepackage{amssymb,amsmath,latexsym}
\usepackage{pst-all}
\usepackage[varg]{pxfonts}
\usepackage{mathrsfs}
\newtheorem{theorem}{Theorem}[section]
\newtheorem{lemma}[theorem]{Lemma}
\newtheorem{corollary}[theorem]{Corollary}

\newtheorem{example}[theorem]{Example}
\newtheorem{proposition}[theorem]{Proposition}

\def\endproof{\qed\endtrivlist}
\expandafter\let\csname endproof*\endcsname=\endproof

\def\qedsymbol{\ifmmode\bgroup\else$\bgroup\aftergroup$\fi
  \vcenter{\hrule\hbox{\vrule height.6em\kern.6em\vrule}\hrule}\egroup}
\def\qed{\ifmmode\else\unskip\nobreak\fi\quad\qedsymbol}

\renewcommand{\iff}{\Leftrightarrow}

\renewcommand{\le}{\leqslant}
\renewcommand{\ge}{\geqslant}
\newcommand{\ind}{\mbox{\rm ind}}

\newcommand{\im}{\qopname\relax{no}{Im}}

\begin{document}

\journal{\qquad}

\title{\Large\bf Weakly linear systems of fuzzy relation inequalities:\\ The heterogeneous case\tnoteref{t1}}
\tnotetext[t1]{Research supported by Ministry  of Education and Science, Republic of Serbia, Grant No. 174013}
\author[fsmun]{Jelena Ignjatovi\'c}
\ead{jelena.ignjatovic@pmf.edu.rs}

\author[fsmun]{Miroslav \'Ciri\'c\corref{cor}}
\ead{miroslav.ciric@pmf.edu.rs}

\author[tfc]{Nada Damljanovi\'c}
\ead{nada@tfc.kg.ac.rs}

\author[fsmun]{Ivana Jan\v ci\'c}
\ead{ivanajancic84@gmail.com}

\cortext[cor]{Corresponding author. Tel.: +38118224492; fax: +38118533014.}
\address[fsmun]{University of Ni\v s, Faculty of Sciences and Mathematics, Vi\v segradska 33, 18000 Ni\v s, Serbia}
\address[tfc]{University of Kragujevac, Technical faculty in \v Ca\v cak, Svetog Save 65, P. O. Box 131, 32000 \v Ca\v cak, Serbia}

\begin{abstract}\small
New types of systems of fuzzy relation inequalities and equations, called weakly linear,
have been recently introduced in [J. Ignjatovi\'c, M. \'Ciri\'c, S. Bogdanovi\'c, On the greatest solutions to weakly linear systems of fuzzy relation inequalities and equations, Fuzzy Sets and Systems 161 (2010)
3081--3113.].~The mentioned paper dealt with homogeneous weakly linear systems, composed of fuzzy relations on a single set, and a method for computing their greatest solutions has been provided.~This method is based on the
computing of the greatest post-fixed point, contained in a given fuzzy~relation, of an
isotone function on the lattice of fuzzy relations.~Here we adapt this method for
computing the greatest solutions of heterogeneous weakly linear systems, where the unknown
fuzzy relation relates two possibly different sets.~We also introduce and study quotient fuzzy relational systems and establish relationships between solutions to heterogeneous and homogeneous weakly linear systems.~Besides, we point out to applications of the obtained results in the state reduction of fuzzy automata and computing the greatest simulations and
bisimulations between fuzzy automata, as well as in the positional analysis of fuzzy social networks.
\end{abstract}


\maketitle

\section{Introduction}

Systems of fuzzy relation equations and inequalities were first studied by Sanchez, who used them~in medical research (cf.~\cite{San.74,San.76,San.77,San.78}).~Later they found a much wider field of application, and nowadays they~are used in fuzzy control, discrete dynamic systems, knowledge engineering, identification of fuzzy systems, predic\-tion of fuzzy systems, decision-making, fuzzy information retrieval, fuzzy pattern recognition, image compres\-sion and reconstruction, and in many other areas (cf., e.g., \cite{DeB.00,DiNSPS.89,DP.80,DP.00,KY.95,PG.07,PK.04}).

Most frequently studied systems were the ones that consist of equations and inequalities with one~side containing the composition of an unknown fuzzy relation and a given fuzzy relation or  fuzzy set,~while the other side contains only another given fuzzy relation or fuzzy set.~Such systems are called {\it linear~systems\/}. Solvability and methods for computing the greatest solutions to linear systems of fuzzy relation equations and inequalities were first studied in the above mentioned papers by Sanchez, who discussed linear~sys\-tems over the G\"odel structure.~Later, linear systems over more general structures of truth values were investigated, including those over complete residuated lattices (cf., e.g., \cite{CIB.09,DeB.00,Klawonn.00,PK.07,Perf.04,PG.03,PN.07}).

More complex non-linear systems of fuzzy relation inequalities and equations, called {\it weakly~linear\/}, have been recently introduced and studied in \cite{ICB.10}.~Basically, weakly linear systems discussed in \cite{ICB.10} consist of inequalities and equations~of the form $V_i\circ U\bowtie U\circ V_i$ and $U\le W$, where $V_i$ ($i\in I$) and $W$ are given fuzzy relations on a set $A$, $U$~is~an~unknown fuzzy relation on $A$, $\circ $ denotes the composition of fuzzy relations, and $\bowtie $ is one of $\le $, $\ge $ and $=$.~Besides, these systems can also include additional inequalities and equations of the form $V_i\circ U^{-1}\bowtie U^{-1}\circ V_i$ and $U^{-1}\le W$, where $U^{-1}$ denotes the converse (inverse, transpose) relation.~Such weakly linear systems, which include only fuzzy relations on a single set, are called {\it homogeneous\/}.

It has been proved in \cite{ICB.10} that each homogeneous weakly linear system, with a complete residuated lattice as the~underlying structure of truth values, has the greatest solution, and an algorithm has been provided for computing this greatest solution.~This algorithm is based on the computing of the greatest post-fixed point, contained in a given fuzzy relation, of an isotone function on the lattice of fuzzy relations, and works whenever the underlying complete residuated lattice is locally finite, for example, when dealing with Boolean~and G\"odel structure.~Otherwise, some sufficient conditions under which the algorithm works have~been deter\-mined.~The mentioned algorithm is iterative, and each its single step can be viewed as solving a particular linear system, and for this reason these systems were called weakly linear.

In this paper we introduce and study heterogeneous weakly linear systems of fuzzy relation inequalities and equations, which are composed of fuzzy relations on two possible different sets, and an unknown is a fuzzy relation between these two sets.~Namely, if $V_i$ and $W_i$ ($i\in I$) are respectively given fuzzy relations on non-empty sets $A$ and $B$, $Z$ is a given fuzzy relation between $A$ and $B$, and $U$ is an unknown fuzzy relation between $A$ and $B$, by a heterogeneous weakly linear system we mean two systems consisting of inequalities of the form $U^{-1}\circ V_i\le W_i\circ U^{-1}$, $U\le Z$, and $V_i\circ U\le U\circ W_i$, $U\le Z$, as well as four systems obtained by combinations of these two systems (for $U$ and $U^{-1}$).

Our main results are the following.~We prove that every heterogeneous weakly linear system has the greatest solution (cf.~Theorem \ref{th:great}). Then we define isotone and image-localized functions $\phi^{(i)}$ ($i=1,\ldots ,6$) on the lattice of fuzzy relations between $A$ and $B$ such that each of the six heterogeneous weakly linear systems can be represented in an equivalent form $U\le \phi^{(i)}(U)$, $U\le Z$ (cf.~Theorem \ref{th:eq.form}).~Such representation enables us to reduce the problem of computing the greatest solution to a heterogeneous weakly linear system to the problem of computing the greatest post-fixed point, contained in the fuzzy relation $Z$, of the function $\phi^{(i)}$. For this purpose we use the iterative method developed in \cite{ICB.10}, adapted to the heterogeneous case.~We also show how the procedure can be modified to compute the greatest crisp solution of the system.~After that, we introduce the concept of the quotient fuzzy relational system with respect to a fuzzy equivalence, and we proved several theorems analogous to the well-known homomorphism, isomorphisms and correspondence theorems from universal algebra (cf.~Theorems \ref{th:E.nat}--\ref{th:F.E}).~Using this concept we establish natural relationships between solutions to heterogeneous and homogeneous weakly linear systems (cf.~Theorems \ref{th:unif}--\ref{th:unif.25}).

The structure of the paper is as follows.~In Section \ref{sec:prel} we give definitions of basic notions and notation concerning fuzzy sets and relations, and in Section \ref{sectionUniform} we recall some notions, notation and results from \cite{CIB.09}, concerning uniform fuzzy relations
and related concepts.~After that, in Section \ref{sec:wlsyst} we recall the definitions of homogeneous weakly linear systems, and then define the heterogeneous ones and prove their fundamental properties.~Section \ref{sec:alg} presents an adaptation of the method developed in \cite{ICB.10} for computing the greatest solution to a heterogeneous weakly linear system.~In Section \ref{sec:quot} we discuss quotient fuzzy relational systems and their main properties, and in Section \ref{sec:relat} we study relationships between solutions to heterogeneous and homogeneous weakly linear systems.~Finally, Section \ref{sec:appl} shows some applications of weakly linear systems in the theory of fuzzy automata and social network analysis.

It is worth noting that weakly linear systems originate from the theory of fuzzy automata.~Solutions to homogeneous weakly linear systems have been used in \cite{CSIP.07,CSIP.10,SCI.11}
for reduction of the number of states, and solutions to the heterogeneous systems have been used in \cite{CIDB.11,CIJD.11} in the study of simulations and bisimulations between fuzzy automata.

\section{Preliminaries}\label{sec:prel}

The terminology and basic notions in this section are according to \cite{Bel.02,BV.05,DP.80,DP.00,KY.95}.

We will use complete residuated lattices as the structures of membership (truth) values.~Residuated lattices are a
very general algebraic structure and generalize many  algebras with very important applications (see for example \cite{Bel.02,BV.05,Hajek.98,Hohle.95}).~A {\it residuated lattice\/} is an algebra ${\cal L}=(L,\land
,\lor , \otimes ,\to , 0, 1)$ such that
\begin{itemize}
\parskip=-2pt
\item[{\rm (L1)}] $(L,\land ,\lor , 0, 1)$ is a lattice with the least element $0$ and the
greatest element~$1$,
\item[{\rm (L2)}] $(L,\otimes ,1)$ is a commutative monoid with the unit $1$,
\item[{\rm (L3)}] $\otimes $ and $\to $ form an {\it adjoint pair\/}, i.e., they satisfy the
{\it adjunction property\/}: for all $x,y,z\in L$,
\begin{equation}\label{eq:adj}
x\otimes y \leqslant z \ \Leftrightarrow \ x\leqslant y\to z .
\end{equation}
\end{itemize}
If, in addition, $(L,\land ,\lor , 0, 1)$ is a complete lattice, then ${\cal L}$ is called a {\it
complete residuated lattice\/}.~Emphasizing their monoidal structure, in some sources residuated lattices are called
integral, commutative, residuated $\ell $-monoids \cite{Hohle.95}.

The operations $\otimes $ (called {\it multiplication\/}) and $\to
$ (called {\it residuum\/}) are intended for modeling the conjunction and implication of the
corresponding logical calculus, and supremum ($\bigvee $) and infimum ($\bigwedge $) are intended
for modeling of the existential and general quantifier, respectively. An operation $\leftrightarrow $ defined
by
\begin{equation}\label{eq:bires}
x\leftrightarrow y = (x\to y) \land (y\to x),
\end{equation}
called {\it biresiduum\/} (or {\it biimplication\/}), is used for modeling the equivalence of truth
values.

The most studied and applied structures of truth values, defined
on the real unit interval $[0,1]$ with $x\land y =\min
(x,y)$ and $x\lor y =\max (x,y)$, are the {\it {\L}ukasiewicz
structure\/} (where $x\otimes y = \max(x+y-1,0)$, $x\to y=
\min(1-x+y,1)$), the {\it Goguen} ({\it product\/}) {\it
structure\/} ($x\otimes y = x\cdot y$, $x\to y= 1$ if $x\le y$,
and~$=y/x$ otherwise), and the {\it G\"odel structure\/} ($x\otimes
y = \min(x,y)$, $x\to y= 1$ if $x\le y$, and $=y$
otherwise).~More~generally, an algebra $([0,1],\land ,\lor ,
\otimes,\to , 0, 1)$ is a complete~resi\-duated lattice if and
only if $\otimes $ is a left-continuous t-norm and the residuum is
defined by $x\to y = \bigvee \{u\in [0,1]\,|\, u\otimes x\le y\}$ (cf.~\cite{BV.05}).~Another
important set of truth values is the set
$\{a_0,a_1,\ldots,a_n\}$, $0=a_0<\dots <a_n=1$, with $a_k\otimes
a_l=a_{\max(k+l-n,0)}$ and $a_k\to a_l=a_{\min(n-k+l,n)}$. A
special case of the latter algebras is the two-element Boolean
algebra of classical logic with the support $\{0,1\}$.~The only
adjoint pair on the two-element Boolean algebra consists of the
classical conjunction and implication operations.~This structure
of truth values we call the {\it Boolean structure\/}.~A
residuated~lattice $\cal L$ satisfying $x\otimes y=x\land y$ is
called a {\it Heyting algebra\/}, whereas a Heyting algebra
satisfying the prelinearity axiom $(x\to y)\lor (y\to x)=1$ is
called a {\it G\"odel algebra\/}. If any finitelly
generated~sub\-algebra of a residuated lattice $\cal L$ is finite,
then $\cal L$ is called {\it locally finite\/}.~For example, every
G\"odel algebra, and hence, the G\"odel structure, is locally
finite, whereas the product structure is not locally finite.

If $\cal L$ is a complete residuated lattice, then for all $x,y,z\in L$ and any $\{y_i\}_{i\in
I}\subseteq L$ the following holds:
\begin{align}
& x\le y\ \ \text{implies}\ \ x\otimes z\le y\otimes z, \label{eq:is.1}\\
& x\le y\ \ \Leftrightarrow\ \ x\to y=1 ,\label{eq:x.le.y} \\
& x\otimes \bigvee_{i\in I}y_i = \bigvee_{i\in I}(x\otimes y_i),\label{eq:prod.sup} \\
& x\otimes \bigwedge_{i\in I}y_i \le  \bigwedge_{i\in I}(x\otimes y_i), \label{eq:prod.inf} \end{align}
For other properties of complete residuated lattices we refer to \cite{Bel.02,BV.05}.

In the further text $\cal L$ will be a complete residuated lattice.~A {\it
fuzzy subset\/} of a set $A$ {\it over\/} ${\cal L}$, or~simply a {\it fuzzy
subset\/} of $A$, is any function from $A$ into $L$.~Ordinary crisp subsets~of~$A$ are considered as fuzzy subsets of $A$ taking membership values in the set
$\{0,1\}\subseteq L$.~Let $f$ and $g$ be two
fuzzy subsets of $A$.~The {\it equality\/} of $f$ and $g$ is defined as the
usual equality of functions, i.e., $f=g$ if and only if $f(x)=g(x)$, for every
$x\in A$. The {\it inclusion\/} $f\leqslant g$ is also defined pointwise:~$f\leqslant g$ if
and only if $f(x)\leqslant g(x)$, for every $x\in A$.~Endowed with this partial order the set ${\cal F}(A)$ of all fuzzy subsets of $A$ forms a complete residuated lattice, in which the
meet (intersection) $\bigwedge_{i\in I}f_i$ and the join (union) $\bigvee_{i\in I}f_i$ of an
arbitrary family $\{f_i\}_{i\in I}$ of fuzzy subsets of $A$ are functions from $A$ into $L$
defined by
\[
\left(\bigwedge_{i\in I}f_i\right)(x)=\bigwedge_{i\in I}f_i(x), \qquad \left(\bigvee_{i\in
I}f_i\right)(x)=\bigvee_{i\in I}f_i(x),
\]
and the {\it product\/} $f\otimes g$ is a fuzzy subset defined by $f\otimes g
(x)=f(x)\otimes g(x)$, for every $x\in A$.~

Let $A$ and $B$ be non-empty sets.~A {\it fuzzy relation between sets\/} $A$ and $B$ (or a {\it fuzzy relation from $A$ to $B$\/}) is any function from $A\times B$ into~$L$, that is to say, any fuzzy subset of $A\times B$, and the equality, inclusion (ordering), joins and meets of fuzzy relations
are defined as for fuzzy sets.~In particular, a {\it fuzzy relation on a set\/} $A$ is any function from $A\times A$ into $L$, i.e., any fuzzy subset of $A\times A$.~The set of all fuzzy relations from~$A$ to $B$ will be denoted by
${\cal R}(A,B)$, and the set of all fuzzy relations on a set $A$ will be
denoted by ${\cal R}(A)$.~The {\it converse\/} (in some sources called
{\it inverse\/} or {\it transpose\/}) of a fuzzy relation $R \in {\cal R}(A,B)$ is a fuzzy relation
$R^{-1}\in {\cal R}(B,A)$ defined by $R^{-1}(b,a)=R (a,b)$, for all $a\in A$ and $b\in B$.~A {\it crisp relation\/} is a fuzzy relation which takes values only in the set $\{0,1\}$, and if $R $ is a crisp relation of $A$ to $B$, then expressions ``$R(a,b)=1$'' and ``$(a,b)\in R $'' will have the same meaning.

For non-empty sets $A$, $B$ and $C$, and fuzzy relations $R\in {\cal R}(A, B)$ and $S \in {\cal R}(B, C)$, their {\it composition\/}~$R\circ S$ is an fuzzy relation from ${\cal R}(A, C)$ defined by
\begin{equation}\label{eq:comp.rr}
(R \circ S )(a,c)=\bigvee_{b\in B}\,R(a,b)\otimes S(b,c),
\end{equation}
for all $a\in A$ and $c\in C$.~If $R $ and $S $ are crisp relations, then $R\circ S $ is the ordinary composition of relations,~i.e.,
\[
R\circ S=\left\{(a,c)\in A\times C\mid (\exists b\in B)\, (a,b)\in R\ \&\ (b,c)\in S\right\},
\]
and if $R $ and $S $ are functions, then $R\circ S $ is an ordinary composition of functions, i.e.,
$(R\circ S)(a)=S(R(a))$, for every $a\in A$.

Let $A$, $B$, $C$ and $D$ be non-empty sets. Then for any $R_1\in {\cal R}(A, B)$, $R_2\in {\cal R}(B,C)$ and $R_3\in {\cal R}(C, D)$~we~have
\begin{align}
&(R_{1}\circ R_{2})\circ R_{3} = R_{1}\circ (R_{2}\circ R_{3}), \label{eq:comp.as}
\end{align}
and consequently, the parentheses~in (\ref{eq:comp.as}) can be omitted, and for $R_{0}\in {\cal R}(A, B)$,\, $R_{1},R_{2}\in {\cal R}(B,C)$ and $R_{3}\in {\cal R}(C, D)$~we~have that
\begin{align}
&R_{1}\leqslant R_{2}\ \ \text{implies}\ \ R_1^{-1}\leqslant R_2^{-1},\ \ R_{0}\circ R_{1} \leqslant  R_{0}\circ R_{2}, \ \ \text{and}\ \ R_{1}\circ R_{3} \leqslant  R_{2}\circ R_{3}. \label{eq:comp.mon}
\end{align}

Finally, for all $R,R_i\in {\cal R}(A, B)$ ($i\in I$) and $S,S_i\in {\cal R}(B, C)$ ($i\in I$) we have that
\begin{align}
&(R\circ S)^{-1} = S^{-1}\circ R^{-1}, \label{eq:comp.inv} \\
& R\circ \bigl(\bigvee_{i\in I}S_i\bigr) = \bigvee_{i\in I}(R\circ S_i) , \ \
\bigl(\bigvee_{i\in I}R_i\bigr)\circ S = \bigvee_{i\in I}(R_i\circ S), \label{eq:comp.sup} \\
&\bigl(\bigvee_{i\in I}R_i\bigr)^{-1} = \bigvee_{i\in I}R_i^{-1}. \label{eq:sup.inv}
\end{align}

We note that if $A$, $B$ and $C$ are finite sets of cardinality $|A|=k$, $|B|=m$ and $|C|=n$, then $R \in {\cal R}(A, B)$~and $S \in {\cal R}(B, C)$ can be treated as $k\times m$ and $m\times n$ fuzzy matrices over $\cal L$, and $R\circ S$ is the matrix~pro\-duct.

A fuzzy relation $E  $ on a set $A$ is%
\begin{itemize}\parskip=0pt
\itemindent=5pt
\item[(R)] {\it reflexive\/} if $E(a,a)=1$, for every $a\in A$;
\item[(S)] {\it symmetric\/} if $E(a,b)=E(b,a)$, for all $a,b\in A$;
\item[(T)] {\it transitive\/} if $E(a,b)\otimes E(b,c)\leqslant E(a,c)$, for all
$a,b,c\in A$.
\end{itemize}
If $E$ is reflexive and transitive, then $E\circ E=E$.
A fuzzy relation on $A$ which is reflexive, symmetric and transitive is called a
{\it fuzzy equivalence relation\/}. With respect to the ordering of fuzzy
relations, the set ${\cal E}(A)$ of all fuzzy equivalence relations on a set $A$ is a complete
lattice, in which the meet coincide with the ordinary intersection of fuzzy relations, but in the
general case, the join in ${\cal E}(A)$ does not coincide with the ordinary union of fuzzy
relations.

For a fuzzy~equi\-valence relation
$E  $ on $A$ and $a\in A$ we define a fuzzy subset $E _a$ of $A$ by:
\[
E _a(x)=E  (a,x),\ \ \text{for every}\ x\in A .
\]
We call $E _a$ an {\it equivalence class\/} of $E$ deter\-mined by $a$.~The set ${A/E}  =\{E _a\,|\, a\in A\}$ is called the {\it
factor set\/} of $A$ with respect to $E$ (cf. \cite{Bel.02,CIB.07}).
Cardinality of the factor set ${A/E}  $, in notation $\ind (E)$, is called the {\it index\/}~of~$E$.
The same notation we use for crisp
equivalence relations, i.e., for an equivalence relation $\pi $ on $A$, the related factor set is denoted
by $A/\pi $, the equivalence class of an element $a\in A$ is denoted by $\pi_a$, and the index~of~$\pi $~is denoted by $\ind (\pi)$.

The following properties of fuzzy equivalence relations will be useful in the later work.

\begin{lemma}\label{le:feq}
Let $E$ be a fuzzy equivalence relation on a set $A$ and let $\widehat E$ be its crisp part.~Then
$\widehat E$ is a crisp~equivalence relation on $A$, and for every $a,b\in A$ the following
conditions are equivalent:
\begin{itemize}\parskip-2pt
\item[{\rm (i)}] $E(a,b)=1$;
\item[{\rm (ii)}] $E_a=E_b$;
\item[{\rm (iii)}] $\widehat E_a=\widehat E_b$;
\item[{\rm (iv)}] $E(a,c)=E(b,c)$, for every $c\in A$.
\end{itemize}
Consequently, $\ind(E)=\ind(\widehat E)$.
\end{lemma}
Note that $\widehat E_a$ denotes the crisp equivalence class of $\widehat E$ determined by $a$.

A fuzzy equivalence relation $E$ on a set $A$ is called a {\it fuzzy
equality\/} if for all $x,y\in A$, $E(x,y)=1$ implies $x=y$. In other words, $E$
is a fuzzy equality if and only if its crisp part $\widehat E$ is a crisp
equality.~For~a~fuzzy~equivalence relation $E$ on a set $A$, a fuzzy relation $\widetilde E$ defined
on the factor set ${A/E}$ by
\[
\widetilde E(E_x,E_y)= E(x,y) ,
\]
for all $x,y\in A$, is well-defined, and it is a fuzzy equality on ${A/E}$.

\section{Uniform fuzzy relations}\label{sectionUniform}

In this section we recall some notions, notation and results from \cite{CIB.09}, concerning
uniform fuzzy relations and related~con\-cepts.

Let $A$ and $B$ be non-empty sets and let $E $ and $F $ be fuzzy equivalence
relations on
$A$ and $B$,
respectively. If a fuzzy relation $R \in {\cal R}(A,B)$ satisfies
\begin{itemize}\itemindent8pt
\item[(EX1)] $R (a_1,b)\otimes E(a_1,a_2)\leqslant R(a_2,b)$, for all $a_1,a_2\in A$ and $b\in B$,
\end{itemize}
then it is called  {\it extensional with respect to\/}~$E$, and if it satisfies
\begin{itemize}\itemindent8pt
\item[(EX2)] $R (a,b_1)\otimes F(b_1,b_2)\leqslant R(a,b_2)$, for all $a\in A$ and $b_1,b_2\in B$,
\end{itemize}
then it is called  {\it extensional with respect to\/}~$F$.~If $R $ is extensional with respect to~$E$ and $F$, and it satisfies
\begin{itemize}\itemindent8pt
\item[(PFF)] $R(a,b_1)\otimes R (a,b_2)\leqslant F(b_1,b_2)$, for all $a\in A$ and $b_1,b_2\in B$,
\end{itemize}
then it is called a {\it partial fuzzy function\/} with respect to~$E$ and $F$.

Partial fuzzy functions were introduced by Klawonn \cite{Klawonn.00}, and studied also by Demirci
\cite{Dem.00,Demirci.03b}. By the adjunction property and symmetry, conditions (EX1) and (EX2) can be also written as
\begin{itemize}\itemindent8pt
\item[(EX1')] $E(a_1,a_2)\leqslant R (a_1,b)\leftrightarrow R(a_2,b)$, for all $a_1,a_2\in A$ and $b\in B$;
\item[(EX2')] $F(b_1,b_2)\leqslant R (a,b_1)\leftrightarrow R(a,b_2)$, for all $a\in A$ and $b_1,b_2\in B$.
\end{itemize}

For any fuzzy relation $R \in {\cal R}(A,B)$ we can define a
fuzzy equivalence relation $E_A^R $ on~$A$~by
\begin{equation}\label{eq:a.phi.1}
E_A^R(a_1,a_2)=\bigwedge_{b\in B} R(a_1,b)\leftrightarrow R (a_2,b),
\end{equation}
for all $a_1,a_2\in A$, and a fuzzy equivalence relation $E_B^R $ on $B$ by
\begin{equation}\label{eq:b.phi.1}
E_B^R(b_1,b_2)=\bigwedge_{a\in A} R(a,b_1)\leftrightarrow R (a,b_2),
\end{equation}
for all $b_1,b_2\in B$.~They will be called {\it fuzzy equivalence relations\/} on $A$ and $B$
{\it induced by\/} $R $, and in particular, $E_A^R $ will be called the {\it kernel\/} of $R $, and
$E_B^R $ the {\it co-kernel\/} of $R $.~According to (EX1')~and (EX2'), $E_A^R $ and $E_B^R $ are the
greatest $\cal $fuzzy equivalence relations on $A$ and $B$, respectively, such that $R $ is
extensional with respect to~them.

A fuzzy relation $R \in {\cal R}(A,B)$ is called just a {\it partial fuzzy function\/} if it is a
partial fuzzy function with respect to $E_A^R $~and~$E_B^R $ \cite{CIB.09}. Partial fuzzy functions
were characterized in \cite{CIB.09} as follows:

\begin{theorem}\label{th:PFF}
Let $A$ and $B$ be non-empty sets and let $R \in {\cal R}(A, B)$ be a fuzzy relation. Then the~fol\-low\-ing conditions are equivalent:
\begin{itemize}\parskip=0pt
\item[{\rm (i)}] $R $ is a partial fuzzy function;
\item[{\rm (ii)}] $R^{-1}$ is a partial fuzzy function;
\item[{\rm (iii)}] $R^{-1}\circ R\leqslant E_B^R $;
\item[{\rm (iv)}] $R \circ R^{-1}\leqslant E_A^R $;
\item[{\rm (v)}] $R \circ R^{-1}\circ R \leqslant R$.
\end{itemize}
\end{theorem}

A fuzzy relation $R \in {\cal R}(A, B)$ is called an $\cal L$-{\it function\/} if
for each $a\in A$ there exists $b\in B$ such that $R (a,b)=1$ \cite{Demirci.05a}, and it is called {\it
surjective\/} if for each $b\in B$ there exists $a\in A$ such that $R (a,b)=1$, i.e., if
$R$ is an $\cal L$-function. For a surjective fuzzy relation $R \in {\cal R}(A, B)$ we also say that it is a fuzzy relation of $A$ {\it onto\/} $B$. If
$R $ is an $\cal L$-function and it is surjective, i.e., if both $R $ and
$R^{-1}$ are $\cal L$-functions, then $R $ is called a {\it surjective $\cal
L$-function\/}.

Let us note that a fuzzy relation $R\in {\cal R}(A, B)$ is an $\cal L$-function if and only if there exists a function $\psi :A\to B$ such that $R (a,\psi(a))=1$, for all $a\in
A$ (cf.~\cite{Demirci.03b,Demirci.05a}). A function $\psi $ with this property we will call a {\it crisp
description\/} of $R $, and we will denote by $CR(R)$ the set of all such functions.

An $\cal L$-function which is a partial fuzzy function with respect to~$E$ and $F$ is called a {\it
perfect~fuzzy function\/} with respect to~$E$ and $F$.~Perfect fuzzy functions were introduced and
studied by Demirci \cite{Dem.00,Demirci.03b}.~A fuzzy relation $R \in
{\cal R}(A, B)$ which is a perfect fuzzy function with respect to $E_A^R $~and~$E_B^R $ will be called just a {\it perfect fuzzy function\/}.

Let $A$ and $B$ be non-empty sets and let $E$ be a fuzzy equivalence relation on $B$.~An ordinary function $\psi :A\to B$ is called {\it $E$-surjective\/} if for any $b\in B$ there exists $a\in A$ such that $E(\psi(a),b)=1$.~In other words,~$\psi $ is $E$-surjective if and only if $\psi \circ E^\sharp $ is an ordinary surjective function
of $A$ onto $B/E$, where $E^\sharp :B\to B/E$~is a function given by $E^\sharp (b)=E_b$,
for each $b\in B$. It is clear that $\psi $ is an $E$-surjective function if and only if~its image $\im \psi $ has a non-empty intersection with every equivalence class of the
crisp equivalence relation~$\ker (E)$.

Let $A$ and $B$ be non-empty sets and let $R\in {\cal R}(A, B)$ be a partial fuzzy function.
If, in addition, $R $ is a surjective $\cal L$-function, then it will be called a {\it
uniform fuzzy relation\/} \cite{CIB.09}. In other words, a uniform fuzzy relation is a perfect fuzzy function
having the additional property that it is surjective.~A uniform fuzzy
relation that~is also a crisp relation is called a {\it uniform relation\/}.~The following characterizations of uniform fuzzy relations provided in \cite{CIB.09} will be used in the further text.

\begin{theorem}\label{th:ufr} Let $A$ and $B$ be non-empty sets and let $R \in {\cal R}(A, B)$
be a fuzzy relation. Then the~follow\-ing conditions are equivalent:
\begin{itemize}\parskip=0pt
\item[{\rm (i)}] $R $ is a uniform fuzzy relation;
\item[{\rm (ii)}] $R^{-1}$ is a uniform fuzzy relation;
\item[{\rm (iii)}] $R $ is a surjective $\cal L$-function and
\begin{equation}\label{eq:3phi}
R\circ R^{-1}\circ R =R ;
\end{equation}
\item[{\rm (iv)}] $R $ is a surjective $\cal L$-function and
\begin{equation}\label{eq:EA.phi}
E_A^R
= R\circ R^{-1};
\end{equation}
\item[{\rm (v)}] $R $ is a surjective $\cal L$-function and
\begin{equation}\label{eq:EB.phi}
E_B^R = R^{-1}\circ R;
\end{equation}
\item[{\rm (vi)}] $R $ is an $\cal L$-function, and for all $\psi\in
CR(R )$, $a\in A$ and $b\in B$ we have that
$\psi $ is $E_B^R $-surjective and
\begin{equation}\label{eq:phi.EB}
R (a,b)=E_B^R (\psi(a),b);
\end{equation}
\item[{\rm (vii)}] $R $ is an $\cal L$-function, and for all $\psi\in
CR(R )$ and $a_1,a_2\in A$ we have that
$\psi $ is $E_B^R $-surjective and
\begin{equation}\label{eq:phi.EA}
R (a_1,\psi(a_2))=E_A^R (a_1,a_2).
\end{equation}
\end{itemize}
\end{theorem}

\begin{corollary}\label{cor:ufr} {\rm \cite{CIB.09}}
Let $A$ and $B$ be non-empty sets, and let $\varphi \in {\cal F}(A\times B)$
be a uniform~fuzzy relation.~Then for all $\psi\in
CR(\varphi )$ and $a_1,a_2\in A$ we have that
\begin{equation}\label{eq:cor.ufr}
E_A^\varphi (a_1,a_2) = E_B^\varphi (\psi(a_1),\psi(a_2)).
\end{equation}
\end{corollary}

Let $A$ and $B$ be non-empty sets. According to Theorem \ref{th:ufr}, a fuzzy relation $R \in {\cal R}(A, B)$ is a uniform fuzzy relation if and only if its inverse relation $R^{-1}$ is a uniform fuzzy relation.~Moreover, by (iv) and (v) of Theorem \ref{th:ufr}, we have that the kernel of $R^{-1}$ is the co-kernel of $R $, and conversely, the co-kernel of $R^{-1}$ is the kernel of $R $, that is
\[
E_B^{R^{-1}}=E_B^R \quad\text{and}\quad E_A^{R^{-1}}=E_A^R .
\]

The next theorem proved in \cite{CIB.09} will be very useful in our further work.

\begin{theorem}\label{th:Rtilde}
Let $A$ and $B$ be  non-empty sets, let $R \in {\cal R}(A, B)$ be a
uniform fuzzy relation,~let $E=E_A^R $ and $F=E_B^R $, and let
a function $\widetilde R :A/E\to B/F$ be defined by
\begin{equation}\label{eq:tphi}
\widetilde R (E_a)= F_{\psi(a)}, \ \ \text{for any $a\in A$ and $\psi \in CR(R )$.}
\end{equation}
Then $\widetilde R $ is a well-defined function {\rm ({\it it does not depend on the choice of $\psi \in CR(R )$ and $a\in A$\/})}, it is a bijective function  of $A/E$ onto $B/F$, and $(\widetilde R)^{-1}=\widetilde{R^{-1}}$.
\end{theorem}

\section{Weakly linear systems}\label{sec:wlsyst}

Now we start with weakly linear systems of fuzzy relation inequalities and equations.~First we recall some concepts from \cite{ICB.10}.

Let $A$ be a non-empty set (not necessarily finite), let $\{V_i\}_{i\in I}$ be a given family of fuzzy relations on $A$ (where $I$ is also not necessarily finite), and let $W$ be a given fuzzy relation on $A$.

In \cite{ICB.10} the following systems of fuzzy relation inequalities and equations have been introduced:
\begin{align}
&U\circ V_i\le V_i\circ U\qquad (i\in I),\qquad\qquad U\le W\,; \label{eq:w11}\tag{$wl$1-1}\\
&V_i\circ U\le U\circ V_i\qquad (i\in I),\qquad\qquad U\le W\,; \label{eq:w12}\tag{$wl$1-2} \\
&U\circ V_i= V_i\circ U\qquad (i\in I),\qquad\qquad U\le W\,; \label{eq:w13}\tag{$wl$1-3}
\end{align}\vspace{-25pt}
\begin{align}
&U\circ V_i\le V_i\circ U,\qquad U^{-1}\circ V_i\le V_i\circ U^{-1},\qquad (i\in I),\qquad\qquad U\le W\,,\ \ U^{-1}\le W\,; \label{eq:w14} \tag{$wl$1-4}\\
&V_i\circ U\le U\circ V_i,\qquad V_i\circ U^{-1}\le U^{-1}\circ V_i,\qquad (i\in I),\qquad\qquad U\le W\,,\ \ U^{-1}\le W\,; \label{eq:w15} \tag{$wl$1-5}\\
&U\circ V_i= V_i\circ U,\qquad U^{-1}\circ V_i= V_i\circ U^{-1}\qquad\ \ (i\in I),\qquad\qquad U\le W\,,\ \ U^{-1}\le W\,; \label{eq:w16}\tag{$wl$1-6}
\end{align}
where $U$ is an unknown fuzzy relation on $A$.~Clearly, a fuzzy relation $R\in {\cal R}(A)$ is a solution to (\ref{eq:w14}) (resp. (\ref{eq:w15}), (\ref{eq:w16})) if and only if both $R$ and $R^{-1}$ are solutions to (\ref{eq:w11}) (resp.~(\ref{eq:w12}), (\ref{eq:w13})), and moreover, a symmetric fuzzy~relation is a solution to (\ref{eq:w14}) (resp.~(\ref{eq:w15}), (\ref{eq:w16})) if and only if it is solution to (\ref{eq:w11}) (resp.~(\ref{eq:w12}), (\ref{eq:w13})).
Clearly, if $W(a_1,a_2)=1$, for all $a_1,a_2\in A$, then the inequality $U\le W$ becomes trivial, and~it~can be omitted.~Systems (\ref{eq:w11})--(\ref{eq:w16}) were called in \cite{ICB.10} {\it weakly linear\/}.~More precisely, in the present paper we will~call them {\it homogeneous weakly linear systems\/}.~For the sake of convenience, for each  $t\in \{1,\ldots ,6\}$, system ($wl1$-$t$) will be denoted by $WL^{\text{1-$t$}}(A,I,V_i,W)$.~If $W(a_1,a_2)=1$, for all $a_1,a_2\in A$, then system ($wl1$-$t$) is denoted simply by $WL^{\text{1-$t$}}(A,I,V_i)$.

It has been proven in \cite{ICB.10} that each of these systems has the greatest solution, and if the given fuzzy relation $W$ is a fuzzy quasi-order, then the greatest solutions to (\ref{eq:w11}), (\ref{eq:w12}) and (\ref{eq:w13}) are also fuzzy quasi-orders, and if $W$ is a fuzzy equivalence, then the greatest solutions to (\ref{eq:w14}), (\ref{eq:w15}) and (\ref{eq:w16}) are fuzzy equivalences.~In the same paper a method for computing these greatest solutions has been developed, and it comes down to the computing of the greatest post-fixed point, contained in a given fuzzy
relation, of an isotone function on the lattice of fuzzy relations.

As we have already said, the purpose of this paper is to introduce the heterogeneous weakly linear~systems, to prove that each of these heterogeneous systems also has the greatest solution (which may be empty), and to show that the method from \cite{ICB.10} can be adapted to compute the greatest solutions to heterogeneous weakly linear systems.

In the further text, let $A$ and $B$ be non-empty sets (not necessarily finite), let $\{V_i\}_{i\in I}$ be a given family of fuzzy relations on $A$ and $\{W_i\}_{i\in I}$ a given family of fuzzy relations on $B$ (where $I$ is also not necessarily finite), and let $Z$ be a given fuzzy relation between $A$ and $B$.

We will consider systems
\begin{align}
&U^{-1}\circ V_i\le W_i\circ U^{-1} &&(i\in I), && U\le Z\,;&&\hspace{60mm} \label{eq:w21}\tag{$wl$2-1} \\
&V_i\circ U\le U\circ W_i && (i\in I), && U\le Z\,;&&\hspace{60mm} \label{eq:w22}\tag{$wl$2-2}
\end{align}
and the systems obtained by combinations of (\ref{eq:w21}) and (\ref{eq:w22}) (for $U$ and $U^{-1}$)
\begin{align}
&U^{-1}\circ V_i\le W_i\circ U^{-1} && U\circ W_i\le V_i\circ U && (i\in I), && U\le Z; \hspace{30mm}\label{eq:w23}\tag{$wl$2-3} \\
&V_i\circ U\le U\circ W_i && W_i\circ U^{-1}\le U^{-1}\circ V_i && (i\in I), && U\le Z;\hspace{30mm}\label{eq:w24}\tag{$wl$2-4} \\
& V_i\circ U= U\circ W_i && {} && (i\in I), && U\le Z; \hspace{30mm} \label{eq:w25}\tag{$wl$2-5}\\
& U^{-1}\circ V_i= W_i\circ U^{-1} && {} && (i\in I), && U\le Z; \hspace{30mm} \label{eq:w26}\tag{$wl$2-6}
\end{align}
where $U$ is an unknown fuzzy relation between $A$ and $B$.~Systems (\ref{eq:w21})--(\ref{eq:w26}) will be called
{\it heterogeneous weakly linear systems\/}.~For the sake of convenience, for each  $t\in \{1,\ldots ,6\}$, system ($wl2$-$t$) will be denoted by
$WL^{\text{2-$t$}}(A,B,I,V_i,W_i,Z)$.~If $Z(a,b)=1$, for all $a\in A$ and $b\in B$, then system ($wl2$-$t$) will be denoted simply by $WL^{\text{2-$t$}}(A,B,I,V_i,W_i)$.

In a very simple way we can prove the following two propositions.

\begin{proposition}\label{prop:dual.hom}
Let $A$ be a non-empty set, let $\{V_i\}_{i\in I}$ be a family of fuzzy relations on $A$, and let $W$ be a fuzzy relation on $A$.~For an arbitrary fuzzy relation $R\in {\cal R}(A)$ the following is true:
\begin{itemize}\parskip=0pt
\item[{\rm (a)}] $R$ is a solution to $WL^{1-1}(A,I,V_i,W)$ if and only if $R^{-1}$ is a solution to $WL^{1-2}(A,I,V_i^{-1},W^{-1})$;
\item[{\rm (b)}] $R$ is a solution to $WL^{1-4}(A,I,V_i,W)$ if and only if $R$ is a solution to $WL^{1-5}(A,I,V_i^{-1},W^{-1})$.
\end{itemize}
\end{proposition}

\begin{proposition}\label{prop:dual.het}
Let $A$ and $B$ be non-empty sets, let $\{V_i\}_{i\in I}$ be a family of fuzzy relations on $A$ and $\{W_i\}_{i\in I}$ a family of fuzzy relations on $B$, and let $Z$ be a fuzzy relation between $A$ and $B$.~For an arbitrary fuzzy relation $R\in {\cal R}(A,B)$ the following is true:
\begin{itemize}\parskip=0pt
\item[{\rm (a)}] $R$ is a solution to $WL^{2-1}(A,B,I,V_i,W_i,Z)$ if and only if it is a solution to $WL^{2-2}(A,B,I,V_i^{-1},W_i^{-1},Z)$;
\item[{\rm (b)}] $R$ is a solution to $WL^{2-3}(A,B,I,V_i,W_i,Z)$ if and only if it is a solution to $WL^{2-4}(A,B,I,V_i^{-1},W_i^{-1},Z)$;
\item[{\rm (c)}] $R$ is a solution to $WL^{2-5}(A,B,I,V_i,W_i,Z)$ if and only if it is a solution to $WL^{2-6}(A,B,I,V_i^{-1},W_i^{-1},Z)$;
\item[{\rm (d)}] $R$ is a solution to $WL^{2-3}(A,B,I,V_i,W_i,Z)$ if and only if $R^{-1}$ is a solution to $WL^{2-4}(B,A,I,W_i,V_i,Z)$.
\end{itemize}
\end{proposition}

The statement (a) in Proposition \ref{prop:dual.hom} says that systems ($wl$1-1) and ($wl$1-2) are dual, in the sense that~for any universally valid statement~on the system ($wl$1-1) there is the corresponding universally valid statement~on the system ($wl$1-2), and vice versa.~Similarly, there is a duality between systems ($wl$1-4) and ($wl$1-5),
($wl$2-1) and ($wl$2-2), ($wl$2-3) and ($wl$2-4), and ($wl$2-5) and ($wl$2-6).~For this reason, in the further text we will deal mainly with the systems ($wl$1-1), ($wl$1-4), ($wl$2-1), ($wl$2-3) and ($wl$2-5).

It is also easy to verify that the following is true.

\begin{proposition}\label{prop:comp}
Let $R$ and $S$ be fuzzy relations such that $R$ is a solution to system $WL^{2-3}(A,B,I,V_i,W_i,Z)$ and $S$ is a solution to system $WL^{2-3}(B,C,I,W_i,X_i,Y)$.
Then $R\circ S$ is a solution to system $WL^{2-3}(A,C,I,V_i,X_i,Z\circ Y)$.
\end{proposition}

Now we prove the following fundamental theorem.

\begin{theorem}\label{th:great}
All heterogeneous weakly linear systems have the greatest solutions (which may be empty).

If $Z$ is a partial fuzzy function, then the greatest solutions to $($\ref{eq:w23}$)$ and $($\ref{eq:w24}$)$ are also partial fuzzy functions.
\end{theorem}

\begin{proof}
For each $t\in \{1,\ldots ,6\}$, it is easy to check that the join (union) of an arbitrary family of fuzzy~relations which are solutions to the system ($wl$2-$t$) is also a solution to ($wl$2-$t$), and consequently, the join of all~solutions to ($wl$2-$t$) is the greatest solution to ($wl$2-$t$).

Next, let $Z$ be a partial fuzzy function and $R$ be the greatest solution to ($wl$2-$t$), where $t=3$ or $t=4$.~Then $R\circ R^{-1}\circ R\le Z\circ Z^{-1}\circ Z\le Z$, and by Proposition \ref{prop:comp} we have that $R\circ R^{-1}\circ R$ is also a solution to ($wl$2-$t$).~As $R$~is~the greatest solution to this system, we conclude that $R\circ R^{-1}\circ R\le R$, which means that $R$ is a partial fuzzy function.
\end{proof}

Next, let $A$ and $B$ be non-empty sets and let $V\in {\cal R}(A)$, $W\in {\cal R}(B)$ and $Z\in {\cal R}(A,B)$.~The {\it right residual\/} of $Z$ by $V$ is a fuzzy relation $Z/V\in {\cal R}(A,B)$ defined by
\begin{equation}\label{eq:rr.def}
(Z/V)(a,b) = \bigwedge_{a'\in A}\, (\, V(a',a) \rightarrow Z(a',b)\, ),
\end{equation}
for all $a\in A$ and $b\in B$, and the {\it left residual\/} of $Z$ by $W$ is a fuzzy relation $Z\backslash W\in {\cal R}(A,B)$ defined by
\begin{equation}\label{eq:lr.def}
(Z\backslash W)(a,b)=  \bigwedge_{b'\in B}\, (\, W(b,b') \rightarrow Z(a,b')\, ),
\end{equation}
for all $a\in A$ and $b\in B$.~In the case when $A=B$, these two concepts become the well-known concepts of~right and left residuals of fuzzy relations on a set (cf.~\cite{ICB.10}).

According to the well-known results by E. Sanchez (cf.~\cite{San.76,San.77,San.78}), the right residual $Z/V$ of $Z$ by $V$ is the greatest solution to the fuzzy relation inequality $V\circ U\le Z$, where $U$ is an unknown fuzzy relation between $A$ and $B$.~Moreover, the set of all solutions to the inequality $V\circ U\le Z$ is the principal ideal~of~${\cal R}(A,B)$~gene\-rated by $Z/V$.~Analogously, the left residual $Z\backslash W$ of $Z$ by $W$ is the greatest solution to the fuzzy relation inequality $U\circ W\le Z$, where $U$ is an unknown fuzzy relation between $A$ and $B$, and the set of all solutions to the inequality $U\circ W\le Z$ is the principal ideal~of~${\cal R}(A,B)$~gene\-rated by $Z\backslash W$.

Consequently, for given families of fuzzy relations $\{V_i\}_{i\in I}\subseteq {\cal R}(A)$, $\{W_i\}_{i\in I}\subseteq {\cal R}(B)$, and $\{Z_i\}_{i\in I}\subseteq {\cal R}(A,B)$, and an unknown fuzzy relation $U$ in ${\cal R}(A,B)$,
the greatest solution to the system $V_i\circ U\le Z_i$ $(i\in I)$ is
\[
\bigwedge_{i\in I} Z_i/V_i ,
\]
i.e., the intersection of the greatest solutions to the individual inequalities $V_i\circ U\le Z_i$, $i\in I$, and the
the greatest solution to the system $U\circ W_i\le Z_i$ $(i\in I)$ is
\[
\bigwedge_{i\in I} Z_i\backslash W_i ,
\]
i.e., the intersection of the greatest solutions to the individual inequalities $U\circ W_i\le Z_i$, $i\in I$.

Define functions $\phi^{(t)}: {\cal R}(A,B)\to {\cal R}(A,B)$,
for $1\le t\le 6$, as follows:
\begin{align}
&\phi^{(1)}(R)=\displaystyle\bigwedge_{i\in I}[(W_i\circ R^{-1})\backslash V_i]^{-1}\label{eq:phi1} \\
&\phi^{(2)}(R)=\displaystyle\bigwedge_{i\in I}(R\circ W_i)/V_i \label{eq:phi2}\\
&\phi^{(3)}(R)=\displaystyle\bigwedge_{i\in I}[(W_i\circ R^{-1})\backslash V_i]^{-1}\land [(V_i\circ R)\backslash W_i] = \phi^{(1)}(R)\land [\phi^{(1)}(R^{-1})]^{-1} \label{eq:phi3}\\
&\phi^{(4)}(R)=\displaystyle\bigwedge_{i\in I}[(R\circ W_i)/V_i]\land [(R^{-1}\circ V_i)/W_i]^{-1}= \phi^{(2)}(R)\land [\phi^{(2)}(R^{-1})]^{-1} \label{eq:phi4}\\
&\phi^{(5)}(R)=\displaystyle\bigwedge_{i\in I}[(R\circ W_i)/V_i]\land [(V_i\circ R)\backslash W_i]=\phi^{(2)}(R)\land [\phi^{(1)}(R^{-1})]^{-1} \label{eq:phi5}\\
&\phi^{(6)}(R)=\displaystyle\bigwedge_{i\in I}[(W_i\circ R^{-1})\backslash V_i]^{-1}\land [(R^{-1}\circ V_i)/W_i]^{-1} = \phi^{(1)}(R)\land [\phi^{(2)}(R^{-1})]^{-1}\label{eq:phi6}
\end{align}
for each $R\in {\cal R}(A,B)$.~Notice that in the expression ``$\phi^{(t)}(R^{-1})$'' ($t\in \{1,2\}$) we denote by $\phi^{(t)}$ a function from ${\cal R}(B,A)$ into itself.

First we show that systems (\ref{eq:w21})--(\ref{eq:w26}) can be represented in equivalent forms, using the functions $\phi^{(t)}$, $1\le t\le 6$, in the following way.

\begin{theorem}\label{th:eq.form}
For every $t\in \{1,\ldots ,6\}$, system ($w2$-$t$) is equivalent to system
\begin{equation}\label{eq:eq.form}
U\le \phi^{(t)}(U), \qquad U\le Z.
\end{equation}
\end{theorem}

\begin{proof}
We will prove only the case $t=1$.~The case $t=2$ is dual to the first one, whereas all other assertions follow by the first two, according to (\ref{eq:phi3})--(\ref{eq:phi6}).

For an arbitrary fuzzy relation $U\in {\cal R}(A,B)$ we have that $U^{-1}\circ V_i\le W_i\circ U^{-1}$ if and only if
\[
U^{-1}(b,a)\otimes V_i(a,a')\le (W_i\circ U^{-1})(b,a'),
\]
for all $a,a'\in A$, $b\in B$ and $i\in I$.~According to the adjunction property, this is equivalent to
\[
U^{-1}(b,a)\le\bigwedge_{a'\in A}[V_i(a,a')\rightarrow (W_i\circ U^{-1}(b,a'))]=((W_i\circ U^{-1})\backslash V_i)(b,a)
\]
for all $a\in A$, $b\in B$ and $i\in I$, which is further equivalent to
\[
U(a,b)\le\bigwedge_{i\in I}[(W_i\circ U^{-1})\backslash V_i]^{-1}(a,b)=(\phi^{(1)}(U))(a,b)
\]
for all $a\in A$ and $b\in B$.~Therefore, $U$ is a solution to ($wl$2-1) if and only if it is a solution to (\ref{eq:eq.form}).
\end{proof}

\section{Computing the greatest solutions}\label{sec:alg}

Here we show that the method developed in \cite{ICB.10}, for computing the greatest solutions to homogeneous weakly linear systems of fuzzy relation inequalities, can be adapted and used for computing the greatest solutions to heterogeneous weakly linear systems.~The mentioned method is based on computing the~greatest post-fixed points of an isotone function on the lattice of fuzzy relations.

Let $A$ and $B$ be non-empty sets  and let $\phi :{\cal R}(A,B)\to {\cal R}(A,B)$ be an isotone~function, which means that $R\le S $ implies $\phi(R)\le \phi (S)$, for all $R ,S\in {\cal R}(A,B)$.~A fuzzy relation $R\in {\cal R}(A,B)$ is called a {\it post-fixed~point\/} of $\phi $ if $R\le \phi(R)$.~The well-known Knaster-Tarski fixed point theorem (stated and proved in a more general context, for complete lattices) asserts that the set of all post-fixed points of $\phi $ form a complete lattice (cf.~\cite{Rom.08}). Moreover, for any fuzzy relation $Z\in {\cal R}(A,B)$ we have that the set of all post-fixed points of $\phi $ contained in $Z$ is non-empty, because it always contains the least element of ${\cal R}(A,B)$ (the empty relation), and it is also a complete lattice.~According to Theorem \ref{th:eq.form}, our main task is to find an effective procedure for computing the greatest post-fixed point of the function $\phi^{(t)}$ contained in the given fuzzy relation $Z$, for each $t\in \{1,\ldots ,6\}$.

Let $\phi :{\cal R}(A,B)\to {\cal R}(A,B)$ be an isotone~function and $Z\in {\cal R}(A,B)$.~We define a sequence $\{R_k\}_{k\in \mathbb{N}}$ of fuzzy relations from ${\cal R}(A,B)$ by
\begin{equation}\label{eq:seq}
R_1=Z, \qquad R_{k+1}=R_k\land \phi(R_k), \ \ \text{for each}\ k\in \mathbb{N} .
\end{equation}
The sequence  $\{R_k\}_{k\in \mathbb{N}}$ is obviously descending.~If we
denote by $\hat R$ the greatest post-fixed point of $\phi$ contained in $Z$, we
can easily verify that
\begin{equation}\label{eq:hat.phi}
\hat R \le \bigwedge_{k\in \mathbb{N}} R_k .
\end{equation}
Now two very important questions arise.~First, under what conditions the equality holds in (\ref{eq:hat.phi})?~Even more important question is: under what conditions the
sequence  $\{R_k\}_{k\in \mathbb{N}}$ is finite? If this sequence is finite,
then it is not hard to show that there exists $k\in \mathbb{N}$ such that $R_k=R_m$, for every $m\ge k$, i.e., there exists $k\in \mathbb{N}$ such that the sequence stabilizes on $R_k$.~We can recognize that the sequence has stabilized when we find the smallest $k\in \mathbb{N}$ such that $R_k=R_{k+1}$.
In this case  $\hat R = R_k$, and we have an algorithm which computes
$\hat R$ in a finite number of steps.

Some conditions under which equality holds in (\ref{eq:hat.phi}) or the sequence
is finite were found in \cite{ICB.10}, in the case which considers fuzzy relations on a single set.~It is not hard to verify that the same results are also valid when fuzzy relations between two sets are considered.~For the sake of completeness we state these results concerning fuzzy relations between two sets.

A sequence $\{R_k\}_{k\in \mathbb{N}}$ of fuzzy relations from ${\cal R}(A,B)$ is called {\it image-finite\/} if the set $\bigcup_{k\in \mathbb{N}}\im (R_k)$ is finite, and it can be easily shown that this sequence is image-finite if and
only if it is finite.~Furthermore,~the~function $\phi :{\cal R}(A,B)\to {\cal R}(A,B)$ is called {\it image-localized\/} if there exists a finite subset $K\subseteq L$ such that for~every fuzzy relation $R\in {\cal R}(A,B)$ we have
\begin{equation}\label{eq:im.loc}
\im (\phi(R))\subseteq \langle K\cup \im (R)\rangle ,
\end{equation}
where $\langle K\cup \im (R)\rangle$ denotes the subalgebra of $\cal L$ generated by the set $K\cup \im (R)$.~Such $K$ will be called a {\it localization set\/} of the function $\phi $.

The following theorem can be proved in the same way as Theorem 5.2 in \cite{ICB.10}.

\begin{theorem}\label{th:phi.loc}
Let the function $\phi $ be image-localized, let $K$ be its localization set, let $Z\in {\cal R}(A,B)$, and let $\{R_k\}_{k\in \mathbb{N}}$ be a sequence of fuzzy relations in ${\cal R}(A,B)$ defined by $(\ref{eq:seq})$. Then
\begin{equation}\label{eq:UimRk}
\bigcup_{k\in \mathbb{N}}\im (R_k) \subseteq \langle K\cup \im (Z)\rangle .
\end{equation}
If, moreover, $\langle K\cup \im (Z)\rangle$ is a finite subalgebra of $\cal L$, then the sequence
$\{R_k\}_{k\in \mathbb{N}}$ is finite.
\end{theorem}

Further we consider $\phi^{(t)}$, for $t\in \{1,\ldots ,6\}$, defined in (\ref{eq:phi1})--(\ref{eq:phi6}).~We prove the following.

\begin{theorem}\label{th:phiw}
All functions $\phi^{(t)}$, for $t\in \{1,\ldots ,6\}$, are isotone, and if $A$, $B$ and $I$ are finite sets, then all these functions are image-localized.
\end{theorem}

\begin{proof}
We will prove only the statement concerning the function $\phi^{(1)}$.~The validity of the statement concerning the function $\phi^{(2)}$ then follows because of duality with the first one, whereas all other statements follow by the first two, according to (\ref{eq:phi3})--(\ref{eq:phi6}).

Let $R_1,R_2\in {\cal R}(A, B)$ such that $R_1\le R_2$, and consider the following systems of fuzzy relation inequalities:
\begin{align}
U^{-1}\circ V_i\le W_i\circ R_1^{-1},\hspace{0.1 in} i\in I;\label{eq:syst.R1} \\
U^{-1}\circ V_i\le W_i\circ R_2^{-1},\hspace{0.1 in} i\in I,\label{eq:syst.R2}
\end{align}
where $U$ is an unknown fuzzy relation between $A$ and $B$.~As we mentioned earlier, fuzzy relations
\[
\phi^{(1)}(R_1)=\bigwedge_{i\in I}[(W_i\circ R_1^{-1})\backslash V_i]^{-1}\quad\text{and}\quad
\phi^{(1)}(R_2)=\bigwedge_{i\in I}[(W_i\circ R_2^{-1})\backslash V_i]^{-1}
\]
are respectively the greatest solutions to (\ref{eq:syst.R1}) and (\ref{eq:syst.R2}), and by $R_1\le R_2$ it follows
$W_i\circ R_1^{-1}\le W_i\circ R_2^{-1}$, for each $i\in I$, so every solution to (\ref{eq:syst.R1}) is a solution to (\ref{eq:syst.R2}).~Consequently, $\phi^{(1)}(R_1)$ is a solution to (\ref{eq:syst.R2}), which means that $\phi^{(1)}(R_1)\le \phi^{(1)}(R_2)$.~Therefore, $\phi^{(1)}$ is an isotone function.

If $A$, $B$ and $I$ are finite sets, then the set $K=\bigcup_{i\in I}(\im(V_i)\cup \im(W_i))$ is also finite, and
for each $R\in {\cal R}(A,B)$ we obviously have that $\im(\phi^{(1)}(R))\subseteq\langle K\bigcup \im(R)\rangle$.~This means that the function $\phi^{(1)}$ is image-localized.
\end{proof}

Now we are ready to state and prove one of the main theorems of the paper.

\begin{theorem}\label{th:alg.fin}
Let $A$, $B$ and $I$ be finite sets, let $\phi =\phi^{(t)}$, for some $t\in\{1,\ldots ,6\}$, and let $\{R_k\}_{k\in \Bbb N}$ be the sequence of fuzzy relations from ${\cal R}(A,B)$ defined by $(\ref{eq:seq})$.

If $\langle \im (Z)\cup \bigcup_{i\in I}(\,\im (V_i)\cup \im(W_i)\,)\rangle $ is a finite subalgebra of $\cal L$, then the following is true:
\begin{itemize}\parskip0pt
\item[{\rm (a)}] the sequence $\{R_k\}_{k\in \mathbb{N}}$ is finite and
descending, and there is the least natural number $k$ such that $R_k=R_{k+1}$;
\item[{\rm (b)}] $R_k$ is the greatest solution to system ($wl$2-t).
\end{itemize}
\end{theorem}

\begin{proof}

Let $\langle \im (Z)\cup \bigcup_{i\in I}(\,\im (V_i)\cup \im(W_i)\,)\rangle $ be a finite subalgebra of $\cal L$.

(a) According to Theorems \ref{th:phiw} and \ref{th:phi.loc}, the sequence $\{R_{k}\}_{k\in N}$ is finite and
descending, so there are~$k,m\in \mathbb{N}$ such that $R_k=R_{k+m}$,
whence $R_{k+1}\le R_k=R_{k+m}\le R_{k+1}$.~Thus, there is $k\in \mathbb{N}$ such that $R_k=R_{k+1}$,~and conse\-quently,
there is the least natural number having this property.

(b) Let $k$ be the least natural number such that $R_k=R_{k+1}$.~It is clear
that $R_k\le Z$.~Moreover, we have that $R_k=R_{k+1}\le \phi^{(t)}(R_k)$,
and according Theorem \ref{th:eq.form} we obtain that $R_k$ is a solution
to the system ($wl$2-t).

Let $R$ be an arbitrary solution to the system ($wl$2-t).~First, we have
that $R\le Z=R_1$.~Next, suppose that $R\le R_m$, for some $m\in \mathbb{N}$.~Then
$R\le \phi^{(t)}(R)\le \phi^{(t)}(R_m)$, so $R\le R_m\land \phi^{(t)}(R_m)=R_{m+1}$.~Therefore,
by induction we conclude that $R\le R_m$, for every $m\in \mathbb{N}$, and
consequently, $R\le R_k$.~Hence, we have proved that $R_k$ is the greatest
solution to the system ($wl$2-t).
\end{proof}

Next, we will consider the case when ${\cal L}=(L,\land ,\lor , \otimes ,\to , 0,
1)$ is a complete residuated lattice satisfying
the following conditions:
\begin{eqnarray}\label{eq:infd}
x\lor \bigl(\bigwedge_{i\in I}y_i\bigr) = \bigwedge_{i\in I}(x\lor y_i) ,\\
\label{eq:infdm}
x\otimes \bigl(\bigwedge_{i\in I}y_i\bigr) = \bigwedge_{i\in I}(x\otimes y_i) ,
\end{eqnarray}
for all $x\in L$ and $\{y_i\}_{i\in I}\subseteq L$.~Let us note that if ${\cal L}=([0,1],\land ,\lor , \otimes ,\to , 0, 1)$, where $[0,1]$ is the real unit interval
and $\otimes $ is a left-con\-tin\-uous t-norm on $[0,1]$, then (\ref{eq:infd}) follows
immediately by linearity of $\cal L$, and $\cal L$ satisfies (\ref{eq:infdm}) if and only if $\otimes $ is
a continuous t-norm, i.e., if and only if $\cal L$ is a $BL$-algebra~(cf.~\cite{Bel.02,BV.05}).~Therefore,~conditions
(\ref{eq:infd}) and (\ref{eq:infdm}) hold for~every $BL$-algebra on the real unit interval.~In particular,
the {\L}ukasiewicz, Goguen (product)~and~G\"odel structures fulfill {\rm (\ref{eq:infd})} and {\rm (\ref{eq:infdm})}.

Under these conditions we have the following.

\begin{theorem}\label{th:infimum}
Let $\phi =\phi^{(t)}$, for some $t\in\{1,\ldots ,6\}$, let $\{R_k\}_{k\in \Bbb N}$ be the sequence of fuzzy relations from ${\cal R}(A,B)$ defined by $(\ref{eq:seq})$, and let $\cal L$ be a complete residuated lattice satisfying
$(\ref{eq:infd})$ and $(\ref{eq:infdm})$.

Then the fuzzy relation
\[
R=\bigwedge_{k\in \mathbb{N}}R_k ,
\]
is the greatest solution to system ($wl$2-t).
\end{theorem}

\begin{proof}
We will prove only the case $t=1$.~All other cases can be proved similarly.

~First we recall a claim proved in  \cite{CSIP.10}, which says that if {\rm (\ref{eq:infd})} is satisfied, then for all descending sequences
$\{x_k\}_{k\in \mathbb{N}}, \{y_k\}_{k\in \mathbb{N}}\subseteq L$ we have
\begin{equation}\label{eq:desc.seq}
\bigwedge_{k\in \mathbb{N}}(x_k\lor y_k) = \bigl(\bigwedge_{k\in \mathbb{N}}x_k\bigr)\lor
\bigl(\bigwedge_{k\in \mathbb{N}}y_k\bigr).
\end{equation}
Now, for arbitrary $i\in I$, $a\in A$ and $b\in B$ we have that
\[
\begin{aligned}
\Biggl(\,\bigwedge_{k\in \mathbb{N}}(W_i\circ R_k^{-1})\Biggr)\,(b,a)&=\bigwedge_{k\in \mathbb{N}}(W_i\circ R_k^{-1})(b,a)=\bigwedge_{k\in \mathbb{N}}\Biggl(\,\bigvee_{b'\in
B}W_i(b,b')\otimes R_k^{-1}(b',a)\Biggr) && \\
&= \bigvee_{b'\in B}\Biggl(\,\bigwedge_{k\in \mathbb{N}}W_i(b,b')\otimes R_k^{-1}(b',a)\Biggr) && \text{(by (\ref{eq:desc.seq}))} \\
&= \bigvee_{b'\in B}\Biggl(\,W_i(b,b')\otimes\biggl( \bigwedge_{k\in \mathbb{N}}R_k^{-1}(b',a)\biggr)\Biggr) && \text{(by (\ref{eq:infdm}))} \\
&= \bigvee_{b'\in B}\biggl(\,W_i(b,b')\otimes R^{-1}(b',a)\biggr)
=(W_i\circ R^{-1})(b,a),  &&
\end{aligned}
\]
which means that
\[
\bigwedge_{k\in \mathbb{N}}W_i\circ R_k^{-1}=W_i\circ R^{-1},
\]
for every $i\in I$.~The use of condition (\ref{eq:desc.seq}) is justified by the facts that $B$ is finite, and that
$\{R_k^{-1}(b',a)\}_{k\in \mathbb{N}}$ is a descending sequence, so
$\{W_i(b,b')\otimes R_k^{-1}(b',a)\}_{k\in \mathbb{N}}$ is also a descending sequence.

Next, for all $i\in I$ and $k\in \mathbb{N}$ we have that
\[
R\le R_{k+1}\le \phi^{(1)}(R_k)=[(W_i\circ R_k^{-1})\backslash
V_i]^{-1},
\]
which is equivalent to
\[
R^{-1}\circ V_i\le W_i\circ R_k^{-1}.
\]
As the last inequality holds for every $k\in \mathbb{N}$, we have that
\[
R^{-1}\circ V_i\le \bigwedge_{k\in \mathbb{N}}W_i\circ R_k^{-1}=W_i\circ R^{-1},
\]
for every $i\in I$.~Therefore, $R$ is solution to ($wl$2-1).

Let $S\in {\cal R}(A,B)$ be an arbitrary fuzzy relation which is solution to ($wl$2-1).~According to Theorem~\ref{th:eq.form},
$S\le \phi^{(1)}(S)$ and $S\le Z=R_1$.~By induction we can easily prove that
$S\le R_{k}$ for every $k\in \mathbb{N}$, and therefore, $S\le
R$.~This means that $R$ is the greatest solution to ($wl$2-1).
\end{proof}

In some situations we do not need solutions to systems of fuzzy relation equations and inequalities that are fuzzy relations, but those that are ordinary crisp relations. Moreover, in cases where our algorithms for computing the greatest solutions to weakly linear systems fail to terminate in a finite number of steps, it is reasonable to search for the greatest crisp solutions to these systems.~They can be understood as some~kind of ``approximations'' of the greatest fuzzy solutions.~It has been shown in \cite{ICB.10} that algorithms for computing the greatest fuzzy solutions to homogeneous weakly linear systems can be modified to compute the greatest crisp solutions to these systems.~Exactly the same way of modification is also applicable to heterogeneous weakly linear systems.~Nevertheless, for the sake of completeness, we will present the~method for computing the greatest solutions to heterogeneous weakly linear systems.

Let $A$ and $B$ be non-empty finite sets, and let ${\cal R}^{c}(A,B)$ be the set of all crisp relations from ${\cal R}(A,B)$.~It~is~easy to verify that ${\cal R}^{c}(A,B)$ is a complete sublattice of ${\cal R}(A,B)$, i.e., the meet and the join in ${\cal R}(A,B)$ of an arbitrary family
of crisp relations from ${\cal R}^{c}(A,B)$ are also crisp relations (in
fact, they coincide with the~ordinary~inter\-section and union of crisp relations).~Moreover,
for any fuzzy relation  $R\in {\cal R}(A,B)$ we have that $R^{c}\in {\cal R}^{c}(A,B)$, where $R^{c}$ denotes the \textit{crisp part\/} of a fuzzy relation $R $ (in some sources called the \textit{kernel\/} of $R $), i.e.,~a
function $R^{c}:A\times B\to \{0,1\}$ defined
by $R^{c}(a,b)=1$, if $R(a,b)=1$, and $R^{c}(a,b)=0$, if
$R(a,b)<1$, for all $a\in A$ and $b\in B$. Equivalently,
$R^{c}$ is considered as an ordinary crisp relation between $A$ and $B$~given
by $R^{c}=\{(a,b)\in A\times B\mid R(a,b)=1\}$.

For each function $\phi :{\cal R}(A,B)\to {\cal R}(A,B)$ we define a function $\phi^{c}: {\cal R}^c(A,B)\to {\cal R}^c(A,B)$~by
\[
\phi^{\mathrm{c}}(R)=(\phi(R))^{\mathrm{c}}, \ \ \text{for any}\ R\in {\cal R}^c(A,B).
\]
If $\phi $ is isotone, then it can be easily shown that $\phi^{\mathrm{c}}$ is also an isotone function.

We have that the following is true.

\begin{proposition}\label{prop:crisp.alg}
Let $A$ and $B$ be non-empty finite sets, let\ $\phi :{\cal R}(A,B)\to {\cal R}(A,B)$ be an isotone function and let  $W\in {\cal R}(A,B)$ be a given fuzzy relation.~A~crisp relation $\varrho\in {\cal R}^c(A,B)$ is the greatest  crisp solution in ${\cal R}(A,B)$ to the system
\begin{equation}\label{eq:c.fri}
U\le \phi (U), \qquad\qquad U\le W,
\end{equation}
 if and only if it is the greatest solution in ${\cal R}^c(A,B)$ to the system
\begin{equation}\label{eq:c.cri}
\xi\le \phi^{c} (\xi), \qquad\qquad \xi\le W^c,
\end{equation}
 where $U$ is an unknown fuzzy relation and $\xi $ is an unknown crisp relation.

Furthermore, a sequence $\{\varrho_k\}_{k\in \mathbb{N}}\subseteq {\cal R}(A,B)$ defined by
\begin{equation}\label{eq:c.seq}
\varrho_1=W^c, \ \ \varrho_{k+1}=\varrho_k\land \phi^{c}(\varrho_k), \ \ \text{for every}\ k\in \mathbb{N},
\end{equation}
is a finite descending sequence of crisp relations, and the least
member of this sequence is the greatest solution to the system $(\ref{eq:c.cri})$ in ${\cal R}^c(A,B)$.
\end{proposition}

Taking $\phi $ to be any of the functions $\phi^{(t)}$, for $t\in \{1,\ldots
 ,6\}$, Proposition \ref{prop:crisp.alg} gives algorithms for computing the
greatest crisp solutions to heterogeneous weakly linear systems.~As we have seen~in Proposi\-tion~\ref{prop:crisp.alg},~these algorithms always terminate in a finite number of steps, independently of the proper\-ties of the underlying structure of truth values, and they could be used in cases when algorithms
for computing the greatest fuzzy solutions do not terminate in a finite number
of steps.~However, the next example shows that there~are~cases when heterogeneous
weakly linear systems have non-empty fuzzy solutions, but they do not have non-empty crisp solutions.

\begin{example}\label{ex:first}\rm
Let $\cal L$ be the G\"odel structure, let $A$ and $B$ be sets with
$|A|=3$ and $|B|=2$, and let fuzzy relations $V_1,V_2\in {\cal R}(A)$,
$W_1,W_2\in {\cal R}(B)$, and $Z\in {\cal R}(A,B)$ be represented by the following fuzzy matrices:
\[
V_{1}=\begin{bmatrix}
 1  & 0.3 & 0.4 \\
0.5 &  1  & 0.3 \\
0.4 & 0.6 & 0.7
\end{bmatrix},\quad
V_{2}=\begin{bmatrix}
0.5 & 0.6 & 0.2 \\
0.6 & 0.3 & 0.4 \\
0.7 & 0.7 &  1
\end{bmatrix},\quad
W_{1}=\begin{bmatrix}
 1  & 0.6 \\
0.6 & 0.7
\end{bmatrix},\quad
W_{2}=\begin{bmatrix}
0.6 & 0.6 \\
0.7 &  1
\end{bmatrix},\quad
Z= \begin{bmatrix}
1 & 1 \\
1 & 1 \\
1 & 1
\end{bmatrix}.
\]
Using algorithms based on Theorem \ref{th:alg.fin} we obtain that the greatest solutions to ($wl2$-1)--($wl2$-6) are respectively given by the fuzzy matrices
\[\small
R^{(1)}=\begin{bmatrix}
 1  & 0.7 \\
 1  & 0.7 \\
0.6 &  1
\end{bmatrix}, \ \
R^{(2)}=\begin{bmatrix}
 1  & 0.7 \\
 1  & 0.7 \\
0.7 &  1
\end{bmatrix}, \ \
R^{(3)}=\begin{bmatrix}
 1  & 0.6 \\
 1  & 0.6 \\
0.6 &  1
\end{bmatrix}, \ \
R^{(4)}=\begin{bmatrix}
 1  & 0.7 \\
 1  & 0.7 \\
0.7 &  1
\end{bmatrix}, \ \
R^{(5)}=\begin{bmatrix}
 1  & 0.6 \\
 1  & 0.6 \\
0.7 &  1
\end{bmatrix}, \ \
R^{(6)}=\begin{bmatrix}
 1  & 0.7 \\
 1  & 0.7 \\
0.6 &  1
\end{bmatrix}.
\]
On the other hand, using the algorithms for computing the
greatest crisp solutions, we obtain that~there~are~no non-empty crisp solutions to ($wl2$-1)-- ($wl2$-6).
\end{example}

It is worth noting that functions $(\phi^{(t)})^c$, for all $t\in \{1,\ldots,6\}$, can be characterized as follows:
\[
\begin{aligned}
&(a,b)\in (\phi^{(1)})^{c}(\varrho )\ \iff\ \ (\forall i\in I)(\forall a'\in A)\,V_{i}(a,a')\le (W_{i}\circ \varrho^{-1})(b,a'), \\
&(a,b)\in (\phi^{(2)})^{c}(\varrho )\ \iff\ \ (\forall i\in I)(\forall a'\in A)\,V_{i}(a',a)\le (\varrho\circ W_{i})(a',b), \\
&(\phi^{(3)})^{c}(\varrho )=(\phi^{(1)})^{c}(\varrho )\land [(\phi^{(1)})^{c}(\varrho^{-1})]^{-1},\quad \
(\phi^{(4)})^{c}(\varrho )=(\phi^{(2)})^{c}(\varrho )\land [(\phi^{(2)})^{c}(\varrho^{-1})]^{-1}, \\
&(\phi^{(5)})^{c}(\varrho )=(\phi^{(2)})^{c}(\varrho )\land [(\phi^{(1)})^{c}(\varrho^{-1})]^{-1},\quad
 \
(\phi^{(6)})^{c}(\varrho )=(\phi^{(1)})^{c}(\varrho )\land [(\phi^{(2)})^{c}(\varrho^{-1})]^{-1},
\end{aligned}
\]
for all $\varrho\in {\cal R}^c(A,B)$, $a\in A$ and $b\in B$.

\section{Quotient fuzzy relational systems}\label{sec:quot}

Loosely speaking, a \emph{relational system\/} is a pair $(A,{\cal R})$ consisting of a non-empty set $A$ and a non-empty family $\cal R$ of finitary relations on $A$ which may have different arities.~Two relational systems $(A,{\cal R}_1)$ and~$(B,{\cal R}_2)$ are considered to be of the same type if a bijective function between ${\cal R}_1$ and ${\cal R}_2$ is given that preserves arity. When we deal only with binary relations, then relational systems $(A,{\cal R}_1)$ and  $(B,{\cal R}_2)$ are of the same type if ${\cal R}_1$ and ${\cal R}_2$ can be written as
${\cal R}_1=\{V_i\}_{i\in I}$ and ${\cal R}_2=\{W_i\}_{i\in I}$, for some non-empty index set~$I$.~In this case, the bijective function that we have mentioned above is just the function that maps $V_i$ to $W_i$, for each $i\in I$.

Here we consider relational systems in the fuzzy context, and we work only with binary fuzzy relations. We define a \emph{fuzzy relational system\/} as a pair ${\cal A}=(A,\{V_i\}_{i\in I})$, where $A$ is a non-empty set and $\{V_i\}_{i\in I}$ is~a~non-empty family of fuzzy relations on $A$, and by fuzzy relational systems of the same type we will mean~systems of the form ${\cal A}=(A,\{V_i\}_{i\in I})$ and ${\cal B}=(B,\{W_i\}_{i\in I})$.~To avoid writing multiple indices, the fuzzy relational system ${\cal A}=(A,\{V_i\}_{i\in I})$ will be sometimes denoted by ${\cal A}=(A,I,V_i)$.~All fuzzy relational systems discussed in the sequel will be of the same type.

Let ${\cal A}=(A,I,V_i)$ and ${\cal B}=(B,I,W_i)$ be two fuzzy relational systems.~A function $\phi :A\to B$ is called an {\it isomorphism\/} if
it is bijective and $V_i(a_1,a_2)=W_i(\phi(a_1),\phi(a_2))$, for all $a_1,a_2\in A$ and $i\in I$.

Let ${\cal A}=(A,I,V_i)$ be a fuzzy relational system and let $E$ be a fuzzy equivalence on $A$.~For each $i\in I$, define a fuzzy relation $V_i^{A/E}$ on the quotient (factor) set $A/E$ as follows:
\begin{equation}\label{eq:quot.syst}
V^{A/E}_i(E_{a_1},E_{a_2})=(E\circ V_i\circ E)(a_1,a_2),
\end{equation}
for all $a_1,a_2\in A$.~The right side of (\ref{eq:quot.syst}) can be equivalently written as
\[
(E\circ V_i\circ E)(a_1,a_2)=\bigvee_{a_1',a_2'\in A}E(a_1,a_1')\otimes V_i(a_1',a_2')\otimes E(a_2',a_2)=E_{a_1}\circ V_i\circ E_{a_2},
\]
and for all $a_1,a_2,a_1',a_2'\in A$ such that $E_{a_1}=E_{a_1'}$ and $E_{a_2}=E_{a_2'}$ we have that $(E\circ V_i\circ E)(a_1,a_2)=(E\circ V_i\circ E)(a_1',a_2')$.
Therefore, the fuzzy relation $V^{A/E}_i$ is well-defined, and ${\cal A}/E=(A/I,I,V^{A/E}_i)$ is a fuzzy relational system of~the~same type as ${\cal A}$, which is called the {\it quotient\/} (or {\it factor\/}) {\it fuzzy relational system\/} of ${\cal A}$, with respect to the fuzzy equivalence $E$.

Note that this concept of quotient fuzzy relational system emerges from the theory of fuzzy automata, namely, it originates from the concept of a factor (quotient) fuzzy automaton.~Factor fuzzy automata were introduced in \cite{CSIP.07,CSIP.10}, where they were used to reduce the number of states of fuzzy automata.~We~will~see~in Section \ref{sec:appl} that quotient (fuzzy) relational systems can be also used to reduce the number of nodes of a (fuzzy) network, while keeping the basic structure of the network.~It is also worth noting that quotient crisp relational systems have been recently defined in the same way in \cite{CL.10}.

The following theorem can be conceived as an analogue of the well-known theorems of universal~algeb\-ra which establish correspondences between functions and equivalence relations, as well as between homomorphisms and congruences (cf.~\cite[\S\,2.6]{BS.81}).

\begin{theorem}\label{th:E.nat}
Let ${\cal A}=(A,I,V_i)$ be a fuzzy relational system, $E$ a fuzzy equivalence on $A$, and ${\cal A}=(A/E,I,V^{A/E}_i)$  the quotient fuzzy relational system of ${\cal A}$ with respect to $E$.

Then a fuzzy relation $E^\natural \in {\cal R}(A, A/E)$ defined by
\begin{equation}\label{eq:E.nat}
E^\natural (a_1,E_{a_2})=E(a_1,a_2), \qquad\qquad \text{for all $a_1,a_2\in A$},
\end{equation}
is a uniform $F$-function whose kernel is $E$.

Moreover, $E^\natural $ is a solution both to $WL^{\text{2-1}}(A,A/E,I,V_i,V^{A/E}_i)$~and  $WL^{\text{2-2}}(A,A/E,I,V_i,V^{A/E}_i)$.
\end{theorem}

\begin{proof}
According to Theorem 7.1 \cite{CIB.09}, $E^\natural $ is a uniform $F$-function of $A$ onto $A/E$ and its kernel is $E$.

Next, for the sake of simplicity set $E^\natural =R$.~Then for all $i\in I$ and $a_1,a_2\in A$ we have that
\begin{equation}\label{eq:E.nat1}
\begin{aligned}
(R^{-1}\circ V_i)(E_{a_1},a_2)&=\bigvee_{a_3\in A}R^{-1}(E_{a_1},a_3)\otimes V_i(a_3,a_2) = \bigvee_{a_3\in A}E(a_1,a_3)\otimes V_i(a_3,a_2) = (E\circ V_i)(a_1,a_2) \\
&\le (E\circ V_i\circ E)(a_1,a_2)=(E\circ V_i\circ E\circ E)(a_1,a_2)=\bigvee_{a_4\in A}(E\circ V_i\circ E)(a_1,a_4)\otimes E(a_4,a_2) \\
&= \bigvee_{a_4\in A}V_i^{A/E}(E_{a_1},E_{a_4})\otimes R^{-1}(E_{a_4},a_2)=(V_i^{A/E}\circ R^{-1})(E_{a_1},a_2),
\end{aligned}
\end{equation}
so $R=E^\natural $ is a solution to the system $WL^{\text{2-1}}(A,A/E,I,V_i,V^{A/E}_i)$, and also,
\begin{equation}\label{eq:E.nat2}
\begin{aligned}
(V_i\circ R)(a_1,E_{a_2})&=\bigvee_{a_3\in A}V_i(a_1,a_3)\otimes R(a_3,E_{a_2})  = \bigvee_{a_3\in A}V_i(a_1,a_3)\otimes E(a_3,a_2) = (V_i\circ E)(a_1,a_2) \\
&\le (E\circ V_i\circ E)(a_1,a_2)=(E\circ E\circ V_i\circ E)(a_1,a_2)=\bigvee_{a_4\in A}E(a_1,a_4)\otimes (E\circ V_i\circ E)(a_4,a_2) \\
&= \bigvee_{a_4\in A}R(a_1,E_{a_4})\otimes V_i^{A/E}(E_{a_4},E_{a_2})=(R\circ V_i^{A/E})(a_1,E_{a_2}),
\end{aligned}
\end{equation}
and hence, $R=E^\natural $ is a solution to the system $WL^{\text{2-2}}(A,A/E,I,V_i,V^{A/E}_i)$.
\end{proof}

We also have the following.

\begin{theorem}\label{th:E.nat2}
Let ${\cal A}=(A,I,V_i)$ be a fuzzy relational system, $E$  a fuzzy equivalence on $A$, and  ${\cal A}=(A/E,I,V^{A/E}_i)$  the quotient fuzzy relational system of ${\cal A}$ with respect to $E$.~Then the following conditions are equivalent:
\begin{itemize}\parskip0pt
\item[{\rm (i)}] $E$ is a solution to $WL^{\text{1-4}}(A,I,V_i)$;
\item[{\rm (ii)}] $E^\natural$ is a solution to $WL^{\text{2-3}}(A,A/E,I,V_i,V^{A/E}_i)$;
\item[{\rm (iii)}] $E^\natural$ is a solution to $WL^{\text{2-5}}(A,A/E,I,V_i,V^{A/E}_i)$.
\end{itemize}
\end{theorem}

\begin{proof}
(i)$\iff $(ii). By Theorem \ref{th:E.nat}, $R=E^\natural$ is a solution to $WL^{\text{2-3}}(A,A/E,I,V_i,V^{A/E}_i)$ if and only if $R\circ V_i^{A/E}\le V_i\circ R$, and according to (\ref{eq:E.nat2}), this is valid if and only if $E\circ V_i\circ E\le V_i\circ E$. On the other hand, since $E\circ V_i\le E\circ V_i\circ E$ and $E\circ E=E$, we have that $E\circ V_i\circ E\le V_i\circ E$ is equivalent to $E\circ V_i\le V_i\circ E$.~Since $E$ is symmetric, we
have that it is a solution to  $WL^{\text{1-4}}(A,I,V_i)$.~Hence, (i)$\iff $(ii) is true.

In the same way we prove that (i)$\iff $(iii).
 \end{proof}

The next theorem can be conceived as an analogue of the well-known Second Isomorphism Theorem from universal algebra (cf.~\cite[\S\,2.6]{BS.81}).

\begin{theorem}\label{th:G:E}
Let ${\cal A}=(A,I,V_i)$ be a fuzzy relational system, let $E$ and $F$ be fuzzy equivalences on $A$ such that~$E\leqslant F$, and let ${\cal A}/E=(A/E,I,V^{A/E}_i)$ be the quotient fuzzy relational system of $\cal A$ with respect to $E$.~Then
a fuzzy relation $F/E$ on $A/E$ defined by
\begin{equation}\label{eq:G:E}
F/E (E_{a_1},E_{a_2})=F(a_1,a_2),\ \ \ \text{for all $a_1,a_2\in A$,}
\end{equation}
is a fuzzy equivalence on $A/E$, and the quotient fuzzy relational systems $({\cal A}/E)/(F/E)$
and ${\cal A}/F$ are isomorphic.
\end{theorem}

\begin{proof}
First we note that $F/E$ is a well-defined fuzzy relation.~Indeed, if $a_1,a_1',a_2,a_2'\in A$ such that
$E_{a_1}=E_{a_1'}$ and $E_{a_2}=E_{a_2'}$, then $E(a_1,a_1')=1=E(a_2,a_2')$, so $F(a_1,a_1')=1=F(a_2,a_2')$, and hence,
$F/E(E_{a_1},E_{a_1'})=F/E(E_{a_2},E_{a_2'})$. It is easy to verify that $F/E$ is a fuzzy equivalence.

For the sake of simplicity set $Q=F/E$, and define a function $\phi :A/G\to (A/E)/Q$ by $\phi (F_a)=Q_{E_a}$, for each $a\in A$.~For arbitrary $a_1,a_2\in A$ we have that
\[
F_{a_1}=F_{a_2}\ \ \iff\ \ F(a_1,a_2)=1\ \ \iff\ \ F/E(E_{a_1},E_{a_2})=1\ \ \iff\ \ Q(E_{a_1},E_{a_2})=1\ \ \iff\ \ Q_{E_{a_1}}=Q_{E_{a_2}},
\]
and thus, $\phi $ is a well-defined and injective function.~It is also clear that $\phi $ is a surjective function.

Furthermore, $E\le F$ yields $F\circ E=E\circ F=F$, and for arbitrary $a_1,a_2\in A$ and $i\in I$ we have that
\[
\begin{aligned}
V_i^{(A/E)/Q}(\phi(F_{a_1}),\phi(F_{a_2}))&= V_i^{(A/E)/Q}(Q_{E_{a_1}},Q_{E_{a_2}})=(Q\circ V_i^{A/E}\circ Q)(E_{a_1},E_{a_2}) \\
&=\bigvee_{a_3,a_4\in A}Q(E_{a_1},E_{a_3})\otimes V_i^{A/E}(E_{a_3},E_{a_4})\otimes Q(E_{a_4},E_{a_2}) \\
&=\bigvee_{a_3,a_4\in A}F({a_1},{a_3})\otimes (E\circ V_i\circ E)({a_3},{a_4})\otimes F({a_4},{a_2}) \\
&= (F\circ E\circ V_i\circ E\circ F)({a_1},{a_2}) = (F\circ V_i\circ F)({a_1},{a_2})=V_i^{A/G}(F_{a_1},F_{a_2}),
\end{aligned}
\]
so we have proved that $\phi $ is an isomorphism of fuzzy relational systems $({\cal A}/E)/(F/E)$
and ${\cal A}/F$.
\end{proof}

We also prove an analogue of the Correspondence Theorem from universal algebra (cf.~\cite[\S\,2.6]{BS.81}).

\begin{theorem}\label{th:emb}
Let ${\cal A}=(A,I,V_i)$ be a fuzzy relational system and let $E$ be fuzzy equivalence on $A$.

The function $\Phi :{\cal E}_E(A)\to {\cal E}(A/E)$, where ${\cal E}_E=\{F\in {\cal E}(A)\mid E\subseteq F\}$, defined by
\begin{equation}\label{eq:emb1}
\Phi (F)=F/E, \quad\text{for all}\ F\in {\cal E}_E(A),
\end{equation}
is an order embedding of ${\cal E}_E(A)$ into ${\cal E}(A/E)$, i.e.,
\begin{equation}\label{eq:emb2}
F\le G \ \ \iff\ \ \Phi(F)\le \Phi (G), \quad \text{for all}\ F,G\in {\cal E}_E(A).
\end{equation}
\end{theorem}

\begin{proof}
For arbitrary $F,G\in {\cal E}_E(A)$ we have that
\[
\begin{aligned}
F\le G\ \ &\iff\ \ (\forall a_1,a_2\in A)\ F(a_1,a_2)\le G(a_1,a_2)\\
&\iff\ \ (\forall a_1,a_2\in A)\ \Phi(F)(E_{a_1},E_{a_2})\le \Phi(G)(E_{a_1},E_{a_2})\ \ \iff\ \ \Phi(F)\le \Phi(G),
\end{aligned}
\]
and hence, $\Phi $ is an order embedding of ${\cal E}_E(A)$ into ${\cal E}(A/E)$.
\end{proof}

It is worth noting that in the case of Boolean (crisp) relational systems $\Phi $ is also surjective, which means that it is an order isomorphism, and equivalently, a lattice isomorphism of ${\cal E}_E(A)$ into ${\cal E}(A/E)$.~In the case of fuzzy relational systems we are not able to prove that fact, but this is not so important because in practice we usually use just the fact that $\Phi $ is an order embedding.

The following theorem will be also very useful in our further work.

\begin{theorem}\label{th:F.E}
Let ${\cal A}=(A,I,V_i)$ be a fuzzy relational system, let $E$ and $F$ be fuzzy equivalences on $A$ such that~$E\leqslant F$, and let ${\cal A}/E=(A/E,I,V^{A/E}_i)$ be the quotient fuzzy relational system of $\cal A$ with respect to $E$.

A fuzzy relation $F_{E}\in {\cal R}(A,A/E)$ defined by
\begin{equation}\label{eq:F:E.nat}
F_E(a_1,E_{a_2})=F(a_1,a_2),\ \ \ \text{for all $a_1,a_2\in A$,}
\end{equation}
is a uniform fuzzy relation with the kernel $F$ and the co-kernel $F/E$.

In addition, if $E$ is a solution to $WL^{\text{1-4}}(A,I,V_i,W)$, for some
$W\in {\cal R}(A)$,
then the following is true:
\begin{itemize}\parskip=0pt
\item[{\rm (a)}] $F$ is a solution to $WL^{\text{1-4}}(A,I,V_i,W)$ if and only if $F/E$ is a solution to $WL^{\text{1-4}}(A/E,I,V_i^{A/E},W/E)$.
\item[{\rm (b)}] $F$ is the greatest solution to  system $WL^{\text{1-4}}(A,I,V_i,W)$ if and only if $F/E$ is the greatest solution to  system $WL^{\text{1-4}}(A/E,I,V_i^{A/E},W/E)$.
\item[{\rm (c)}] $F$ is a solution to $WL^{\text{1-4}}(A,I,V_i,W)$ if and only if $F_E$ is a solution to  $WL^{\text{2-3}}(A,A/E,I,V_i,V^{A/E}_i,W_{E})$.
\end{itemize}
\end{theorem}

\begin{proof}
For the sake of simplicity set $F_E=G$.~For arbitrary $a_1,a_2\in A$ we can easily check that
\[
\begin{aligned}
&(G\circ G^{-1}\circ G)(a_1,E_{a_2})=(F\circ F^{-1}\circ F)(a_1,{a_2})=F(a_1,a_2)=G(a_1,E_{a_2}), \\
&(G\circ G^{-1})(a_1,{a_2})=(F\circ F^{-1})(a_1,{a_2})=F(a_1,a_2), \\
&(G^{-1}\circ G)(E_{a_1},E_{a_2})=(F^{-1}\circ F)(a_1,{a_2})=F(a_1,a_2)=F/E(E_{a_1},E_{a_2}),
\end{aligned}
\]
which means that $G\circ G^{-1}\circ G=G$, $G\circ G^{-1}=F$ and $G^{-1}\circ G=F/E$.~Therefore, $G$ is a uniform fuzzy relation with the kernel $F$ and the co-kernel $F/E$.

Next, let $E$ be a solution to $WL^{1-4}(A,I,V_i)$.

(a) According to Theorem \ref{th:emb}, $F\le W$ if and only if $F/E\le W/E$.~Furthermore,
since $E\le F$ is equivalent to $E\circ F=F\circ E=F$, for arbitrary $a_1,a_2\in A$ and $i\in I$ we have that
\[
\begin{aligned}
&(F/E)\circ V_i^{A/E}(E_{a_1},E_{a_2})=(F\circ E\circ V_i\circ E)(a_1,a_2)= (F\circ V_i\circ E)(a_1,a_2), \\
&V_i^{A/E}\circ (F/E)(E_{a_1},E_{a_2})=(E\circ V_i\circ E\circ F)(a_1,a_2)= (E\circ V_i\circ F)(a_1,a_2),
\end{aligned}
\]
and as $F/E$ is symmetric, then it is  a solution to $WL^{\text{1-4}}(A/E,I,V_i^{A/E},W/E)$ if and only if
\begin{equation}\label{eq:F.E}
F\circ V_i\circ E\le E\circ V_i\circ F,\quad F\le W
\end{equation}
for each $i\in I$.~Therefore, it remains to prove that $F$ is a solution to $WL^{\text{1-4}}(A,I,V_i,W)$ if and only if (\ref{eq:F.E})~holds. Bearing in mind that $E$ is a solution to $WL^{\text{1-4}}(A,I,V_i,W)$, if $F$ is also a solution to this system, then $F\le W$ and $F\circ V_i\circ E\le V_i\circ F\circ E=V_i\circ F$ and $E\circ V_i\circ F\le V_i\circ E\circ F=V_i\circ F$, so we have that (\ref{eq:F.E}) is true.~Conversely, let (\ref{eq:F.E}) hold.~Then
\[
F\circ V_i\le F\circ V_i\circ E\le E\circ V_i\circ F\le V_i\circ E\circ F=V_i\circ F,
\]
for each $i\in I$, which means that $F$ is a solution to $WL^{\text{1-1}}(A,I,V_i,W)$.

(b) Let $F$ be the greatest solution to  the system $WL^{\text{1-4}}(A,I,V_i,W)$.~Assume
that $Q$ is the greatest solution to $WL^{\text{1-4}}(A/E,I,V_i^{A/E},W/E)$, and define a fuzzy relation $G$ on $A$ as follows:
\[
G(a_1,a_2)=Q(E_{a_1},E_{a_2}), \quad \text{for all}\ a_1,a_2\in A.
\]
It is easy to check that $G$ is a fuzzy equivalence on $A$.~According to the assertion (a) of this theorem, $E/E$ is a solution to $WL^{\text{1-4}}(A/E,I,V_i^{A/E},W/E)$, so $E/E\le Q$.~Now, for arbitrary $a_1,a_2\in A$ we have that
\[
E(a_1,a_2)=E/E(E_{a_1},E_{a_2})\le Q(E_{a_1},E_{a_2})=G(a_1,a_2),
\]
which means that $E\le G$, and consequently, $Q=G/E$.~Next, by the assertion (a) of this theorem we obtain that $G$ is a solution to $WL^{\text{1-4}}(A,I,V_i,W)$, and since $F$ is the greatest solution to this system, then $G\le F$. According to Theorem \ref{th:emb}, $Q=G/E\le F/E$, and since $F/E$ is a solution to $WL^{\text{1-4}}(A/E,I,V_i^{A/E},W/E)$ and~$Q$ is the greatest solution to this system, we have that $Q=F/E$, i.e., $F/E$ is the greatest solution to $WL^{\text{1-4}}(A/E,I,V_i^{A/E},W/E)$.

Conversely, let $F/E$ be the greatest solution to $WL^{\text{1-4}}(A/E,I,V_i^{A/E},W/E)$.~According to (a), $F$ is a solution to the system $WL^{\text{1-4}}(A,I,V_i,W)$.~Let $G$ be the greatest  solution to $WL^{\text{1-4}}(A,I,V_i,W)$.
By Theorem 4.5 \cite{ICB.10},~$G$ is a fuzzy equivalence, and we have that $E\le F\le G$.~Next, by (a) we obtain that
$G/E$ is a solution to $WL^{\text{1-4}}(A/E,I,V_i^{A/E},W/E)$, so $G/E\le F/E$.~But, now by Theorem \ref{th:emb} it follows that $G\le F$, i.e., $G=F$, so we have proved that $F$ is the greatest solution to $WL^{\text{1-4}}(A,I,V_i,W)$.

(c) For arbitrary $a_1,a_2\in A$ and $i\in I$ we have that
\[
\begin{aligned}
&(F_E^{-1}\circ V_i)(E_{a_1},a_2)=(F\circ V_i)(a_1,a_2), \ \
(V_i^{A/E}\circ F_E^{-1})(E_{a_1},a_2)=(E\circ V_i\circ E\circ F)(a_1,a_2)=(E\circ V_i\circ F)(a_1,a_2),\\
&(F_E\circ V_i^{A/E})(a_1,E_{a_2}) =(F\circ E\circ V_i\circ E)(a_1,a_2)=(F\circ V_i\circ E)(a_1,a_2),\ \ (V_i\circ F_E)(a_1,E_{a_2})=(V_i\circ F)(a_1,a_2),
\end{aligned}
\]
so $F_E$ is a solution to  $WL^{\text{2-3}}(A,A/E,I,V_i,V^{A/E}_i,W_{E})$ if and only if
$F\circ V_i\le E\circ V_i\circ F$ and $F\circ V_i\circ E\le V_i\circ F$, for each $i\in I$, and $F\le W$.~It is easy to verify that $F\circ V_i\circ E\le V_i\circ F$ is equivalent to $F\circ V_i\le V_i\circ F$ even if $E$ is not a solution to $WL^{1-4}(A,I,V_i,W)$ (using only reflexivity of $E$ and the equality $F\circ E=F$).~On the other hand, under assumption that $E$ is a solution to $WL^{1-4}(A,I,V_i,W)$ we obtain that $F\circ V_i\le E\circ V_i\circ F$ is also equivalent to $F\circ V_i\le V_i\circ F$.~Thus, we have proved that (c) is true.
\end{proof}

\section{Relationships between solutions to heterogeneous and homogeneous weakly linear systems}\label{sec:relat}

In this section we determine the relationships between solutions to heterogeneous and homogeneous weakly linear systems.~In particular, we show that the kernel and the co-kernel of a solution to a heterogeneous weakly linear system are solutions to
related homogeneous systems, and we establish the connection between the greatest solutions to a heterogeneous systems and the related homogeneous systems.

First we prove the following.

\begin{proposition}\label{prop:RR-1}
Let a fuzzy relation $R\in {\cal R}(A,B)$ be a solution to system $WL^{\text{2-3}}(A,B,I,V_i,W_i,Z)$.~Then
\begin{itemize}\parskip0pt
\item[{\rm (a)}] $R\circ R^{-1}$ is a solution to system $WL^{\text{1-4}}(A,I,V_i,Z\circ Z^{-1})$;
\item[{\rm (b)}] $R^{-1}\circ R$ is a solution to system $WL^{\text{1-4}}(B,I,W_i,Z^{-1}\circ Z)$.
\end{itemize}
\end{proposition}

\begin{proof}
For each $i\in I$, by $R^{-1}\circ V_i\le W_i\circ R^{-1}$ and $R\circ W_i\le V_i\circ R$ it follows that
\[
R\circ R^{-1}\circ V_i\le R\circ W_i\circ R^{-1}\le V_i\circ R\circ R^{-1}\quad\text{and}\quad
R^{-1}\circ R\circ W_i\le R^{-1}\circ V_i\circ R\le W_i\circ R^{-1}\circ R,
\]
and by $R\le Z$ we obtain that $R\circ R^{-1}\le Z\circ Z^{-1}$ and $R^{-1}\circ R\le Z^{-1}\circ Z$.~Since $R\circ R^{-1}$ and $R^{-1}\circ R$ are symmetric fuzzy relations, we have that  $R\circ R^{-1}$ is a solution to $WL^{\text{1-4}}(A,I,V_i,Z\circ Z^{-1})$ and $R^{-1}\circ R$ is a solution to $WL^{\text{1-4}}(B,I,W_i,Z^{-1}\circ Z)$.
\end{proof}

In the previous proposition we have considered the solution of system (\ref{eq:w23}) which is an arbitrary fuzzy relation.~In the following theorem we deal with solutions to this system which are uniform fuzzy relations.

\begin{theorem}\label{th:unif}
Let $R\in {\cal R}(A,B)$ be a uniform fuzzy relation and let $Z\in {\cal R}(A,B)$ be a fuzzy relation
such that $R\le Z$.

Then $R$ is a solution to system $WL^{\text{2-3}}(A,B,I,V_i,W_i,Z)$ if and only if the following is true:
\begin{itemize}\parskip0pt
\item[{\rm (i)}] $E_A^R$ is a solution to system $WL^{\text{1-4}}(A,I,V_i,Z\circ Z^{-1})$;
\item[{\rm (ii)}] $E_B^R$ is a solution to system $WL^{\text{1-4}}(B,I,W_i,Z^{-1}\circ Z)$;
\item[{\rm (iii)}] $\widetilde R$ is an isomorphism of quotient fuzzy relational systems ${\cal A}/E_A^R $ and ${\cal B}/E_B^R $;
\end{itemize}
where ${\cal A}=(A,I,V_i)$ and ${\cal B}=(B,I,W_i)$.
\end{theorem}

\begin{proof}
For the sake of simplicity set $E_A^R=E$, $E_B^R=F$ and $\widetilde R=\phi $.~By uniformity of $R$ we have that $E=R\circ R^{-1}$ and $F=R^{-1}\circ R$.

Let $R$ be a solution to $WL^{\text{2-3}}(A,B,I,V_i,W_i,Z)$.~By Proposition \ref{prop:RR-1}, (i) and (ii) hold.~According to Theorem \ref{th:Rtilde}, $\phi $ is a bijective function of $A/E$ onto $B/F$. Moreover, for an arbitrary $i\in I$ we have that
\[
\begin{aligned}
E\circ V_i\circ E&=R\circ R^{-1}\circ V_i\circ R\circ R^{-1}\le R\circ W_i\circ R^{-1}\circ R\circ R^{-1} = R\circ W_i\circ R^{-1}\\
&= R\circ W_i\circ R^{-1} = R\circ R^{-1}\circ R\circ W_i\circ R^{-1} \le  R\circ R^{-1}\circ V_i\circ R\circ R^{-1} = E\circ V_i\circ E,
\end{aligned}
\]
and hence, $E\circ V_i\circ E=R\circ W_i\circ R^{-1}$.

Further, for arbitrary $a_1,a_2\in A$, $i\in I$ and $\psi\in CR(R)$ we have that
\[
\begin{aligned}
V_i^{A/E}(E_{a_1},E_{a_2})&= (E\circ V_i\circ E)(a_1,a_2)=(R\circ W_i\circ R^{-1})(a_1,a_2) \\
&=\bigvee_{b_1,b_2\in B}R(a_1,b_1)\otimes W_i(b_1,b_2)\otimes R^{-1}(b_1,a_2)
=\bigvee_{b_1,b_2\in B}F(\psi(a_1),b_1)\otimes W_i(b_1,b_2)\otimes F(b_1,\psi(a_2))\\
&=(F\circ W_i\circ F)(\psi(a_1),\psi(a_2))=W_i^{B/F}(F_{\psi(a_1)},F_{\psi(a_2)})= W_i^{B/F}(\phi(E_{a_1}),\phi(E_{a_2})).
\end{aligned}
\]
Thus, $\phi $ is an isomorphism of fuzzy relational systems ${\cal A}/E$ and ${\cal B}/F$.

Conversely, let (i), (ii) and (iii) hold.~Consider arbitrary $\varphi \in CR(R)$, $\psi\in CR(R^{-1})$, $a_1,a_2\in A$, $b_1,b_2\in B$ and $i\in I$.~Then we have that
\[
(E\circ V_i\circ  E)(a_1,a_2)=V_i^{A/E}(E_{a_1},E_{a_2})=W_i^{B/F}(\phi(E_{a_1}),\phi(E_{a_2})) =W_i^{B/F}(F_{\varphi(a_1)},F_{\varphi(a_2)})= (F\circ W_i\circ  F)(\varphi(a_1),\varphi(a_2)),
\]
and similarly,
\[
(F\circ W_i\circ  F)(b_1,b_2)=(E\circ V_i\circ  E)(\psi(b_1),\psi(b_2)).
\]
Now, for arbitrary $a\in A$, $b\in B$, $i\in I$, $\varphi \in CR(R)$ and $\psi\in CR(R^{-1})$ we have that
\[
\begin{aligned}
&(R^{-1}\circ V_i)(b,a)=(R^{-1}\circ E\circ V_i)(b,a)\le (R^{-1}\circ V_i\circ E)(b,a) = \bigvee_{a_1,a_2\in A}R^{-1}(b,a_1)\otimes V_i(a_1,a_2)\otimes E(a_2,a) \\
&\qquad\qquad= \bigvee_{a_1,a_2\in A}E(\psi(b),a_1)\otimes V_i(a_1,a_2)\otimes E(a_2,a) = (E\circ V_i\circ E)(\psi(b),a)= (F\circ W_i\circ F)(\varphi(\psi(b)),\varphi(a)) \\
&\qquad\qquad= \bigvee_{b_1,b_2\in B}F(\varphi(\psi(b)),b_1)\otimes W_i(b_1,b_2)\otimes F(b_2,\varphi(a)),
\end{aligned}
\]
and since $F(\varphi(\psi(b)),b)=R(\psi(b),b)=R^{-1}(b,\psi(b))=1$ implies $F(\varphi(\psi(b)),b_1)=F_{\varphi(\psi(b))}(b_1)=F_{b}(b_1)=F(b,b_1)$, we obtain that
\[
\begin{aligned}
&\bigvee_{b_1,b_2\in B}F(\varphi(\psi(b)),b_1)\otimes W_i(b_1,b_2)\otimes F(b_2,\varphi(a)) = \bigvee_{b_1,b_2\in B}F(b,b_1)\otimes W_i(b_1,b_2)\otimes F(b_2,\varphi(a)) \\
&\qquad\qquad= (F\circ W_i\circ F)(b,\varphi(a))\le (W_i\circ F)(b,\varphi(a))=\bigvee_{b_3\in B}W_i(b,b_3)\otimes F(b_3,\varphi (a)) \\
&\qquad\qquad= \bigvee_{b_3\in B}W_i(b,b_3)\otimes R^{-1}(b_3,a)= (W_i\circ R^{-1})(b,a).
\end{aligned}
\]
Hence, $R^{-1}\circ V_i\le W_i\circ R^{-1}$, and in a similar way we prove that $R\circ W_i\le V_i\circ R$. This completes the proof of the theorem.
\end{proof}

A natural question which arises here is the relationship between the greatest solution to a heterogeneous weakly linear system and the greatest solutions to the corresponding homogeneous weakly linear systems. The following theorem gives an answer to this question.

\begin{theorem}\label{th:unif.gr}
Let $Z\in {\cal R}(A,B)$ be a uniform fuzzy relation and let system $WL^{\text{2-3}}(A,B,I,V_i,W_i,Z)$ have a uniform solution.

Then the greatest solution $R$ to $WL^{\text{2-3}}(A,B,I,V_i,W_i,Z)$ is a uniform fuzzy relation such that
$E_A^R$ is the greatest solution to $WL^{\text{1-4}}(A,I,V_i,Z\circ Z^{-1})$ and $E_B^R$ is the greatest solution to $WL^{\text{1-4}}(B,I,W_i,Z^{-1}\circ Z)$.
\end{theorem}

\begin{proof}
According to Theorem \ref{th:great}, $R$ is a partial fuzzy function, and since the system $WL^{\text{2-3}}(A,B,I,V_i,W_i,Z)$ has a uniform solution, this uniform solution is contained in $R$, so $R$ is also a uniform fuzzy relation.

For the sake of simplicity set $E_A^R=E$, $E_B^R=F$ and $\widetilde R=\phi $.~By Theorem \ref{th:unif} it follows that $E$ is a solution
to $WL^{\text{1-4}}(A,I,V_i,Z\circ Z^{-1})$, $F$ is a solution to $WL^{\text{1-4}}(B,I,W_i,Z^{-1}\circ Z)$ and $\phi$ is an isomorphism of quotient fuzzy relational systems
${\cal A}/E$ and ${\cal B}/F$, where ${\cal A}=(A,I,V_i)$ and ${\cal B}=(A,I,W_i)$.

Furthermore, assume that $G$ is the greatest solution to $WL^{\text{1-4}}(A,I,V_i,Z\circ Z^{-1})$ and $H$ is the greatest solution to $WL^{\text{1-4}}(B,I,W_i,Z^{-1}\circ Z)$.~Let $S=G_E$ and $T=H_F$, where $G_E\in {\cal R}(A,A/E)$ and $H_F\in {\cal R}(B,B/F)$~are~fuzzy relations defined as in (\ref{eq:F:E.nat}).~According to Theorem \ref{th:F.E}, we have that $S$ and $T$ are uniform fuzzy relations such that $E_A^S=G$, $E_{A/E}^S=G/E$, $E_B^T=H$ and $E_{B/F}^T=H/F$.~The same theorem asserts that $S$ is a solution to $WL^{\text{2-3}}(A,A/E,I,V_i,V_i^{A/E},P_E)$ and $T$ is a solution to $WL^{\text{2-3}}(B,B/F,I,W_i,W_i^{B/F},Q_F)$, where $P=Z\circ Z^{-1}$ and $Q=Z^{-1}\circ Z$.~Moreover, if we consider the isomorphism $\phi $ as a fuzzy relation between $A/E$ and $B/F$, then it is easy to verify that $\phi $ is a solution to $WL^{\text{2-3}}(A/E,B/F,I,V_i^{A/E},W_i^{B/F},\phi)$.

Now, let a fuzzy relation $M\in {\cal R}(A,B)$ be defined as $M=S\circ \phi \circ T^{-1}$.~According to Propositions \ref{prop:dual.het} (d) and \ref{prop:comp}, $M$ is a solution to system $WL^{\text{2-3}}(A,B,I,V_i,W_i,P_E\circ \phi\circ Q_F^{-1})$.~We will prove that $P_E\circ \phi\circ Q_F^{-1}=Z$.

Consider arbitrary $a\in A$ and $b\in B$.~First, we have that
\[
(P_E\circ \phi\circ Q_F^{-1})(a,b)=\bigvee_{a_1\in A}P_E(a,E_{a_1})\otimes (\phi\circ Q_F^{-1})(E_{a_1},b) .
\]
Moreover, for arbitrary $a_1\in A$ and $\psi \in CR(R)$ we obtain that
\[
(\phi\circ Q_F^{-1})(E_{a_1},b)=\bigvee_{b_1\in B}\phi(E_{a_1},F_{b_1})\otimes Q_F^{-1}(F_{b_1},b) = Q_F^{-1}(F_{\psi(a_1)},b)= Q_F(b,F_{\psi(a_1)})=Q(b,\psi(a_1)),
\]
and since $Z$ is a uniform fuzzy relation and $Q=Z^{-1}\circ Z=E_B^Z$, then  $Q(b,\psi(a_1))=E_B^Z(\psi(a_1),b)=Z(a_1,b)$, according to Theorem \ref{th:ufr}.~Therefore,
\[
(P_E\circ \phi\circ Q_F^{-1})(a,b)=\bigvee_{a_1\in A}P_E(a,E_{a_1})\otimes Z(a_1,b) =
\bigvee_{a_1\in A}(Z\circ Z^{-1})(a,a_1)\otimes Z(a_1,b)=(Z\circ Z^{-1}\circ Z)(a,b)=Z(a,b),
\]
and we have proved that $P_E\circ \phi\circ Q_F^{-1}=Z$.~Hence, $M$ is a solution to
$WL^{\text{2-3}}(A,B,I,V_i,W_i,Z)$, and since $R$ is the greatest solution to this system, we conclude that $M\le R$.

Further, consider arbitrary $a\in A$ and $\psi \in CR(R)$.~Then
\[
\begin{aligned}
M(a,\psi(a))&= (S\circ \phi \circ T^{-1})(a,\psi(a))=\bigvee_{a_1\in A}S(a,E_{a_1})\otimes (\phi\circ T^{-1})(E_{a_1},\psi(a)) \\
&= \bigvee_{a_1\in A}S(a,E_{a_1})\otimes \biggl( \bigvee_{b\in B}(\phi (E_{a_1},F_b)\otimes T^{-1}(F_b,\psi(a))\biggr) = \bigvee_{a_1\in A}S(a,E_{a_1})\otimes T^{-1}(F_{\psi(a_1)},\psi(a)) \\
&= \bigvee_{a_1\in A}G(a,a_1)\otimes H(\psi(a),\psi(a_1))\ge G(a,a)\otimes H(\psi(a),\psi(a))=1,
\end{aligned}
\]
and consequently,
\[
(M\circ M^{-1})(a,a)=\bigvee_{b\in B}M(a,b)\otimes M^{-1}(b,a)\ge M(a,\psi(a))\otimes M^{-1}(\psi(a),a)=1.
\]
Hence, $M\circ M^{-1}$ is reflexive.~As $G=E_A^S=S\circ S^{-1}$, we have that $G\circ M=S\circ S^{-1}\circ S\circ \phi \circ T^{-1}=
S\circ \phi \circ T^{-1}=M$, and by reflexivity of $M\circ M^{-1}$ we obtain that $G\le G\circ M\circ M^{-1}=M\circ M^{-1}\le R\circ R^{-1}=E$.~Since both $G$ and $E$ are solutions to system
$WL^{\text{1-4}}(A,I,V_i,Z\circ Z^{-1})$, and $G$ is the greatest one, we conclude that $E=G$, i.e., $E=E_A^R$ is the greatest solution to $WL^{\text{1-4}}(A,I,V_i,Z\circ Z^{-1})$.

In the same way we show that $F=H$, i.e., $F=E_B^R$ is the greatest solution to $WL^{\text{1-4}}(B,I,W_i,Z^{-1}\circ Z)$.~This completes the proof of the theorem.
\end{proof}

Let us note that the fuzzy relation $M$ defined in the proof of the previous theorem can be also represented as $M=G\circ \psi \circ H$, for an arbitrary $\psi\in CR(R)$.

A result similar to Theorem \ref{th:unif} can  be also obtained for system (\ref{eq:w25}).

\begin{theorem}\label{th:unif.25}
Let $R\in {\cal R}(A,B)$ be a uniform fuzzy relation and $Z\in {\cal R}(A,B)$ is a fuzzy relation
such that $R\le Z$.

Then $R$ is a solution to system $WL^{\text{2-5}}(A,B,I,V_i,W_i,Z)$ if and only if the following is true:
\begin{itemize}\parskip0pt
\item[{\rm (i)}] $E_A^R$ is a solution to system $WL^{\text{1-4}}(A,I,V_i,Z\circ Z^{-1})$;
\item[{\rm (ii)}] $E_B^R$ is a solution to system $WL^{\text{1-5}}(B,I,W_i,Z^{-1}\circ Z)$;
\item[{\rm (iii)}] $\widetilde R$ is an isomorphism of quotient fuzzy relational systems ${\cal A}/E_A^R $ and ${\cal B}/E_B^R $;
\end{itemize}
where ${\cal A}=(A,I,V_i)$ and ${\cal B}=(B,I,W_i)$.
\end{theorem}

\begin{proof}
For the sake of simplicity set $E_A^R=E$, $E_B^R=F$ and $\widetilde R=\phi $.~By uniformity of $R$ we have that $E=R\circ R^{-1}$ and $F=R^{-1}\circ R$.

Let $R$ be a solution to $WL^{\text{2-5}}(A,B,I,V_i,W_i,Z)$.~Due to reflexivity of $E$ and $F$, for each $i\in I$ we have
\[
\begin{aligned}
E\circ V_i&\le E\circ V_i\circ E = R\circ R^{-1}\circ V_i\circ R\circ R^{-1}=
R\circ R^{-1}\circ R\circ W_i\circ R^{-1} \\
&=R\circ W_i\circ R^{-1}=V_i\circ R\circ R^{-1}=V_i\circ E, \\
W_i\circ F&\le F\circ W_i\circ F =R^{-1}\circ R\circ W_i\circ R^{-1}\circ R
=R^{-1}\circ V_i\circ R\circ R^{-1}\circ R \\
&= R^{-1}\circ V_i\circ R = R^{-1}\circ R\circ W_i = F\circ W_i ,
\end{aligned}
\]
and by symmetry of $E$ and $F$ we obtain that $E$ is a solution to $WL^{\text{1-4}}(A,I,V_i,Z\circ Z^{-1})$ and $F$ is a solution to  $WL^{\text{1-5}}(B,I,W_i,Z^{-1}\circ Z)$.~As we have shown above, $E\circ V_i\circ E = R\circ W_i\circ R^{-1}$, for every $i\in I$, and as in the proof of Theorem \ref{th:unif} we prove that $\phi $ is an isomorphism of fuzzy relational systems ${\cal A}/E$ and ${\cal B}/F$.

Conversely, let (i), (ii) and (iii) hold.~As in the proof of Theorem \ref{th:unif} we show that
\[
(E\circ V_i\circ E)(a_1,a_2)=(F\circ W_i\circ F)(\varphi(a_1),\varphi(a_2)), \qquad
(F\circ W_i\circ F)(b_1,b_2)=(E\circ V_i\circ E)(\psi(b_1),\psi(b_2)),
\]
for all $a_1,a_2\in A$, $b_1,b_2\in B$, $i\in I$, $\varphi \in CR(R)$ and $\psi \in CR(R^{-1})$.~Thus, for all $a\in A$, $b\in B$ and $i\in I$ we have that
\[
\begin{aligned}
(V_i\circ R)(a,b)&=(V_i\circ E\circ R)(a,b)=(E\circ V_i\circ E\circ R)(a,b)=(E\circ V_i\circ R)(a,b) \\
&=\bigvee_{a_1\in A}(E\circ V_i)(a,a_1)\otimes R(a_1,b)= \bigvee_{a_1\in A}(E\circ V_i)(a,a_1)\otimes E(a_1,\psi(b))=(E\circ V_i\circ E)(a,\psi(b)) \\
&= (F\circ W_i\circ F)(\varphi(a),\varphi(\psi(b)))=
\bigvee_{b_1\in B}(F\circ W_i)(\varphi(a),b_1)\otimes F(b_1,\varphi(\psi(b)))=(F\circ W_i\circ F)(\varphi(a),b)\\
&=(F\circ W_i)(\varphi(a),b)=\bigvee_{b_2\in B}F(\varphi(a),b_2)\otimes W_i(b_2,b) =
\bigvee_{b_2\in B}R(a,b_2)\otimes W_i(b_2,b) = (R\circ W_i)(a,b).
\end{aligned}
\]
Therefore, $V_i\circ R=R\circ W_i$, for each $i\in I$, and we have proved that $R$ is a solution to $WL^{\text{2-5}}(A,B,I,V_i,W_i,Z)$.
\end{proof}

It is an open question whether the analogue of Theorem \ref{th:unif.gr} is valid for the system (\ref{eq:w25}).~The methodology used in Theorem \ref{th:unif.gr} does not give results when it works with this system.

\section{Some applications}\label{sec:appl}

Fuzzy relational systems have many natural interpretations and important applications.~We will mention two of them, and we will also point to applications of weakly linear systems related to these interpretations.

First, a fuzzy relational system ${\cal A}=(A,I,V_i)$ can be interpreted as the system of fuzzy transition relations of some {\it fuzzy transition system\/} \cite{CCK.11} or a {\it fuzzy automaton\/} (when fuzzy sets of initial and terminal states are also fixed) \cite{CIDB.11,CIJD.11,CSIP.07,CSIP.10,SCI.11} with $A$ as the set of states and $I$ as the input alphabet (set of labels).~In~this~interpretation,~the concept of a quotient fuzzy relational system corresponds to the concept of a quotient (factor) fuzzy~auto\-maton or fuzzy transition system which has been introduced in \cite{CSIP.07,CSIP.10}.~Quotient fuzzy automata have been used in \cite{CSIP.07,CSIP.10,SCI.11} to~reduce~the number of states of a fuzzy automaton, and from this aspect, there are interesting those quotient fuzzy automata which are language-equivalent to the original fuzzy automaton. In particular, the language-equivalence is achieved when the quotient fuzzy automaton is made by means of fuzzy equivalences which are solutions to homogeneous weakly linear systems (\ref{eq:w14}) and (\ref{eq:w15}) (or (\ref{eq:w11}) and (\ref{eq:w12})), and the best such reductions are attained by means of the greatest solutions to these systems.~Let us note that in these cases the fuzzy relation $W$ is taken to be the greatest fuzzy equivalence such that the fuzzy set of terminal states or the fuzzy set of initial states is extensional with respect to it.

On the other hand, heterogeneous weakly linear systems (\ref{eq:w21})--(\ref{eq:w26}) have been studied in the context of fuzzy automata in \cite{CIDB.11,CIJD.11} (see also \cite{CIBJ.11}), with the fuzzy relation $Z$ given in terms of fuzzy sets of initial and terminal states, and certain additional constraints given also in terms of fuzzy sets of initial and terminal states. Solutions to (\ref{eq:w21}) and (\ref{eq:w22}) are called {\it simulations\/} (respectively {\it forward\/} and {\it backward simulations\/}), and solutions to (\ref{eq:w23})--(\ref{eq:w26}) are called {\it bisimulations\/} (respectively {\it forward\/}, {\it backward\/}, {\it backward-forward\/} and {\it forward-backward bisimulations\/}).~All types of bisimulations realize the language-equivalence between fuzzy automata, and forward and backward bisimulations which are uniform fuzzy relations are used to model structural equivalence between fuzzy automata.~More information on fuzzy automata with membership~values in complete residuated lattices, state reduction, simulation, bisimulation and equivalence can be found in \cite{CIDB.11,CIJD.11,CSIP.07,CSIP.10,SCI.11}.

In another interpretation of the fuzzy relational system ${\cal A}=(A,I,V_i)$, $A$ is taken to be a set of {\it individuals\/} and $\{V_i\}_{i\in I}$ is a system of fuzzy relations between these individuals.~Such a fuzzy relational system~is~called~a {\it fuzzy social network\/}, or just a {\it fuzzy network\/}, since concepts of social network analysis share many common properties with other types of networks and its methods are applicable to the analysis of networks in~general.~In large and complex networks it is impossible to understand~the relationship between each pair of indi\-viduals, but to a certain extent, it may be possible to understand~the system, by classifying individuals and describing relationships on the class level.~In networks, for instance, individuals in the same class can be considered to occupy the same {\it position\/},
or play the same {\it role\/} in the network.~The main aim of the position\-al analysis of networks is~to find similarities between individuals which have to reflect their position in a network.~These similarities have been formalized first by Lorrain and White \cite{LW.71} by the concept of a~{\it structural equiv\-alence\/}.~Informally speaking, two individuals are considered to be structurally equivalent if they have identical neighborhoods.~However, in many situations this concept has shown oneself to be too strong.~Weakening it sufficiently to make it more appropriate for modeling social positions, White and Reitz~\cite{WR.83} have introduced the concept of a {\it regular equivalence\/}, where two~indi\-viduals are considered to be regularly equivalent~if they are equally related to equivalent~others.~In the context of fuzzy relations, regular equivalences correspond to fuzzy equivalences which are solutions to the homogeneous weakly linear system (\ref{eq:w16}) (or (\ref{eq:w13})).~Making the quotient fuzzy relational system with respect to a regular fuzzy equivalence we reduce the number of nodes in the original network, while preserving essential relations between nodes.~Fuzzy relations which are solutions to systems (\ref{eq:w21})--(\ref{eq:w26}) also have natural interpretations when dealing with bipartite networks, which will be the topic of our further research.~More information on various aspects of the network analysis and its applications can be found in \cite{BE.05,HR.05,M.07,dNMB.05,DBF.05}.

\section{Concluding remarks}

New types of fuzzy relation inequalities and equations, called {\it weakly~linear\/}, have been recently introduced and studied in \cite{ICB.10}.~They are composed of fuzzy relations on a single set and are called {\it homogeneous\/}.~In this paper we have introduced and studied {\it heterogeneous weakly linear systems\/}, which are composed of fuzzy relations on two possible different sets, and an unknown is a fuzzy relation between these two sets.~We have proved that every heterogeneous weakly linear system has the greatest solution, and we define isotone and image-localized functions $\phi^{(i)}$ ($i=1,\ldots ,6$) on the lattice of fuzzy relations between $A$ and $B$ such that each of the six heterogeneous weakly linear systems can be represented in an equivalent form $U\le \phi^{(i)}(U)$, $U\le Z$.~Such representation enables us to reduce the problem of computing the greatest solution to a heterogeneous weakly linear system to the problem of computing the greatest post-fixed point, contained in the fuzzy relation $Z$, of the function $\phi^{(i)}$. For this purpose we use the iterative method developed in \cite{ICB.10}, adapted to the heterogeneous case.~Besides, we introduce the concept of the quotient fuzzy relational system with respect to a fuzzy equivalence, we proved several theorems analogous to the well-known homomorphism, isomorphisms and correspondence theorems from universal algebra, and using this concept we establish natural relationships between solutions to heterogeneous and homogeneous weakly linear systems.

Weakly linear systems originate from research in the theory of fuzzy automata.~Solutions to homogeneous weakly linear systems have been used in \cite{CSIP.07,CSIP.10,SCI.11}
for reduction of the number of states, and solutions to the heterogeneous systems have been used in \cite{CIDB.11,CIJD.11} in the study of simulations and bisimulations between fuzzy automata. In our further work both homogeneous and heterogeneous weakly linear systems will be used in the study of fuzzy social networks.~Besides, methodology developed in the study of weakly linear systems will be generalized and applied to a much wider class of fuzzy relation inequalities and equations, as well as to matrix inequalities and equations over max-algebras.


\begin{thebibliography}{33}


\bibitem{Bel.02}
R. B\v elohl\'avek, Fuzzy Relational Systems: Foundations and Principles, Kluwer,
New York, 2002.

\bibitem{BV.05}
R. B\v elohl\'avek, V. Vychodil, Fuzzy Equational Logic, Springer, Berlin/Heidelberg, 2005.

\bibitem{BE.05}
U. Brandes, T. Erlebach (eds.), Network Analysis: Methodological Foundations,
(Lecture Notes in Computer Science, vol. 3418), Springer, 2005.

\bibitem{BS.81}
S. Burris, H. P. Sankappanavar, A Course in Universal Algebra, Springer-Verlag, New York, 1981.

\bibitem{CCK.11}
Y. Cao, G. Chen, E. Kerre, Bisimulations for fuzzy transition systems, IEEE Transactions on Fuzzy Systems (2011), doi:10.1109/TFUZZ.2011.2117431.

\bibitem{CL.10}
I. Chajda, H. L\"anger, Quotients and homomorphisms of relational systems, Acta Univ. Palacki. Olomuc., Fac. rer. nat., Mathematica 49 (2010) 37--47.

\bibitem{CIBJ.11}
M. \'Ciri\'c, J. Ignjatovic, M. Ba\v si\'c, I. Jan\v ci\'c, Nondeterministic automata: Simulation, bisimulation and structural equivalence, submitted to Computers \&\ Mathematics with Applications.

\bibitem{CIB.07}
M. \'Ciri\'c, J. Ignjatovi\'c, S. Bogdanovi\'c, Fuzzy equivalence relations and
their equivalence classes, Fuzzy Sets and Systems 158 (2007) 1295--1313.

\bibitem{CIB.09}
M. \'Ciri\'c, J. Ignjatovi\'c, S.  Bogdanovi\'c, Uniform fuzzy relations and fuzzy functions, Fuzzy Sets and Systems 160 (2009) 1054--1081.


\bibitem{CIDB.11}
M. \'Ciri\'c, J. Ignjatovi\'c, N. Damljanovi\'c, M. Ba\v si\'c, Bisimulations for fuzzy automata, submitted to Fuzzy Sets and Systems.

\bibitem{CIJD.11}
M. \'Ciri\'c, J. Ignjatovi\'c, I. Jan\v ci\'c, N. Damljanovi\'c, Algorithms for computing the greatest simulations and bisimulations between fuzzy automata, submitted to Fuzzy Sets and Systems.

\bibitem{CSIP.07}
M. \'Ciri\'c, A. Stamenkovi\'c, J. Ignjatovi\'c, T. Petkovi\'c, Factorization of
fuzzy automata, In: Csuhaj-Varju, E., \'Esik, Z. (eds.), FCT 2007, Springer, Heidelberg,
Lecture Notes in Computer Science 4639 (2007) 213--225.

\bibitem{CSIP.10}
M. \'Ciri\'c, A. Stamenkovi\'c, J. Ignjatovi\'c, T. Petkovi\'c, Fuzzy relation equations and reduction of fuzzy automata, Journal of Computer and System Sciences 76 (2010) 609--633.

\bibitem{Dem.00}
M. Demirci, Fuzzy functions and their applications,  Journal of Mathematical Analysis and Applications 252 (2000) 495--517.

\bibitem{Demirci.03b}
M. Demirci, Foundations of fuzzy functions and vague algebra based on many-valued equivalence relations, Part I: Fuzzy functions and their applications,
International Journal of general Systems 32 (2) (2003) 123--155.

\bibitem{Demirci.05a}
M. Demirci, A theory of vague lattices based on many-valued equivalence
relations -- I: general representation results,
Fuzzy Sets and Systems 151 (2005) 437--472.


\bibitem{DeB.00}
B. De Baets, Analytical solution methods for fuzzy relational equations, in: D. Dubois, H. Prade (eds.), Fundamentals of Fuzzy Sets, The Handbooks of Fuzzy Sets Series, Vol. 1, Kluwer Academic Publishers, 2000, pp. 291--340.


\bibitem{dNMB.05}
W. De Nooy, A. Mrvar, V. Batagelj, Exploratory Network Analysis with Pajek, Cambridge University Press, 2005.

\bibitem{DiNSPS.89}
A. Di Nola, E. Sanchez, W. Pedrycz, S. Sessa, Fuzzy Relation Equations and Their Application to Knowledge Engineering,
Kluwer Academic Press, Dordrecht, 1989.

\bibitem{DBF.05}
P. Doreian, V. Batagelj, A. Ferligoj, Generalized Blockmodeling, Cambridge University Press, 2005.




\bibitem{DP.80}
D. Dubois, H. Prade, Fuzzy Sets and Systems: Theory and Applications, Academic Press, New York, 1980.

\bibitem{DP.00}
D. Dubois, H. Prade (eds.), Fundamentals of Fuzzy Sets, The Handbooks of Fuzzy Sets
Series, Vol. 1, Kluwer Academic Publishers, 2000.


\bibitem{Hajek.98}
P. H\'ajek, Mathematics of fuzzy logic, Kluwer, Dordrecht, 1998.

\bibitem{HR.05}
R. A. Hanneman, M. Riddle, Introduction to Social Network Methods, University of California, Riverside, 2005.

\bibitem{Hohle.95}
U. H\"ohle,
Commutative, residuated $\ell$-monoids,
in: U. H\"ohle and E. P. Klement (Eds.), Non-Classical Logics and Their
Applications to Fuzzy Subsets, Kluwer Academic Publishers, Boston, Dordrecht,
1995, pp. 53--106.

\bibitem{ICB.10}
J. Ignjatovi\'c, M. \'Ciri\'c, S. Bogdanovi\'c, On the greatest solutions to certain systems of fuzzy relation inequalities and equations, Fuzzy Sets and Systems 161 (2010) 3081--3113.

\bibitem{Klawonn.00}
F. Klawonn, Fuzzy points, fuzzy relations and fuzzy functions,
in: V. Nov\^{a}k and I. Perfilieva (Eds.), Discovering World with Fuzzy Logic,
Physica-Verlag, Heidelberg, 2000, pp. 431--453.

\bibitem{KY.95}
G. J. Klir, B. Yuan, Fuzzy Sets and Fuzzy Logic, Theory and Application, Prentice-Hall, Englevood
Cliffs, NJ, 1995.

\bibitem{LW.71}
F. Lorrain, H. C. White, Structural equivalence of individuals in social networks,
Journal of Mathematical Sociology 1 (1971) 49--80.

\bibitem{M.07}
P. Mika, Social Networks and the Semantic Web, Springer, 2007.

\bibitem{PG.07}
W. Pedrycz, F. Gomide, Fuzzy Systems Engineering: Toward Human-Centric Computing, Wiley-IEEE Press, 2007.

\bibitem{PK.04}
K. Peeva, Y. Kyosev, Fuzzy Relational Calculus: Theory, Applications, and Software (with
CD-ROM), in Series ``Advances in Fuzzy Systems -- Applications and Theory'', Vol 22,
World Scientific, 2004.


\bibitem{PK.07}
K. Peeva, Y. Kyosev, Algorithm for solving max-product fuzzy relational equations, Soft
Computing 11 (2007) 593--605.

\bibitem{Perf.04}
I. Perfilieva, Fuzzy function as an approximate solution to a system of fuzzy relation equations,
Fuzzy Sets and Systems 147 (2004) 363--383.

\bibitem{PG.03}
I. Perfilieva, S. Gottwald, Fuzzy function as a solution to a system of fuzzy relation equations, International Journal of General
Systems 32 (2003) 361--372.

\bibitem{PN.07}
I. Perfilieva, V. Nov\'ak, System of fuzzy relation equations as a continuous model
of IF-THEN rules, Information Sciences 177 (2007) 3218--3227.


\bibitem{Rom.08}
S. Roman, Lattices and Ordered Sets, Springer, New York, 2008.

\bibitem{San.74}
E. Sanchez, Equations de relations floues, Th\`ese de Doctorat, Facult\'e de M\'edecine de Marseille, 1974.

\bibitem{San.76}
E. Sanchez, Resolution of composite fuzzy relation equations, Information and Control 30 (1976) 38--48.

\bibitem{San.77}
E. Sanchez, Solutions in composite fuzzy relation equations: application to medical diagnosis in Brouwerian logic, in: M. M.
Gupta, G. N. Saridis, B. R. Gaines (Eds.), Fuzzy Automata and Decision Processes, North-Holland, Amsterdam, 1977, pp. 221--234.

\bibitem{San.78}
E. Sanchez, Resolution of eigen fuzzy sets equations, Fuzzy Sets and Systems 1 (1978) 69--74.

\bibitem{SCI.11}
A. Stamenkovi\'c, M. \'Ciri\'c, J. Ignjatovi\'c, Reduction of fuzzy automata by means of fuzzy
quasi-orders, submitted to Information Sciences.

\bibitem{WR.83}
D. R. White, K. P. Reitz, Graph and semigroup homomorphisms on networks and relations, Social Networks 5 (1983) 143--234.


\end{thebibliography}
\end{document}